\documentclass[11pt,letter]{article}
\usepackage{amsfonts,epsf,amsmath,amssymb,graphicx,tabularx,booktabs}

\newcommand{\ICG}{\mathrm{ICG}}
\newcommand{\N}{{\mathbb N}}
\newcommand{\Z}{{\mathbb Z}}
\newcommand{\R}{{\mathbb R}}

\newtheorem{theorem}{\bf Theorem}[section]

\newtheorem{lemma}[theorem]{\bf Lemma}
\newtheorem{proposition}[theorem]{\bf Proposition}

\newcommand{\qed}{\hfill $\blacksquare$ \bigskip}

\begin{document}

\title{\bf{Graphs with three and four distinct eigenvalues based on circulants}}

\author {
{\large Milan Ba\v si\' c} \\
{\em Faculty of Sciences and Mathematics, University of Ni\v s, Serbia} \\
{e-mail: \texttt{basic\_milan@yahoo.com}}} 


\date{\today}
\maketitle 

\begin{abstract}

In this paper, we aim to address the open questions raised in various recent papers regarding characterization of circulant graphs with three or four distinct eigenvalues in their spectra.
Our focus is on providing characterizations and constructing classes of graphs falling under this specific category.
We present a characterization of circulant graphs with prime number order and unitary Cayley graphs with arbitrary order, both of which possess spectra displaying three or four distinct eigenvalues.
Various constructions of circulant graphs with composite orders are provided whose spectra consist of four distinct eigenvalues.
These constructions primarily utilize specific subgraphs of circulant graphs that already possess two or three eigenvalues in their spectra, employing graph operations like the tensor product, the union, and the complement.
Finally, we characterize the iterated line graphs of unitary Cayley graphs whose spectra contain three or four distinct eigenvalues, and we show their non-circulant nature.

\end{abstract}

 \vspace{0.2cm}


\section{Introduction}
\label{sec:intro}

Circulant graphs are Cayley graphs over a cyclic group. A graph is
called integral if all the eigenvalues of its adjacency matrix are
integers. In other words, the corresponding adjacency matrix of a
circulant graph is the circulant matrix (a special kind of Toeplitz
matrix where each row vector is rotated one element to the right
relative to the preceding row vector).  Integral graphs are
extensively studied in the literature and there has been a vast
research on some types of classes of graphs with integral spectrum.
The interest for circulant graphs in graph theory and applications
has grown during the last two decades. They appear in coding theory,
VLSI design, Ramsey theory and other areas.  Since they possess many
interesting properties (such as vertex transitivity called mirror
symmetry), circulants are applied in quantum information
transmission and proposed as models for quantum spin networks that
permit the quantum phenomenon called perfect state transfer
\cite{Ba14,Ba15,Ge11}. In the quantum communication scenario, the
important feature of this kind of quantum graphs (especially those
with integral spectrum) is the ability of faithfully transferring
quantum states without modifying the network topology.


Graphs with a high multiplicity of eigenvalues in their adjacency matrix spectrum, indicating a limited number of distinct eigenvalues, are commonly acknowledged to possess a distinctive structure.
In such cases, specific graph operations can often provide effective representations for these graphs.
Particularly, graphs with coinciding eigenvalues are considered trivial, whereas connected graphs with two distinct eigenvalues are known as complete graphs, and connected regular graphs with precisely three distinct eigenvalues are classified as strongly regular.
Nonregular graphs that possess three distinct eigenvalues have received relatively less attention in the existing literature.
This is primarily because their appealing combinatorial properties tend to diminish when the graph loses its regularity, leading to increased complexity. A similar observation can be made for graphs that are regular but have four distinct eigenvalues \cite{Da95,Da98}.
Finding a characterization of connected regular graphs with four distinct eigenvalues is an extremely challenging problem. Even when we focus on particular types of graphs, like circulant graphs, the topic continues to be attractive.


This paper focuses on investigating spectral characteristics of circulant graphs, a topic that has been explored in several recent papers \cite{Ba22, ChRa18, ChRaGu12, SaSa15, KlSa10}.
Additionally, through examination of the properties and characterization of circulant graphs (both integral and non-integral) possessing three or four distinct eigenvalues, we can make a valuable contribution to the spectral theory of both classes of graphs.
Furthermore, by employing graph theoretical operations, such as line operations performed on a class of circulant graphs, one can derive classes of non-circulant integral graphs that exhibit spectra comprising three or four distinct eigenvalues.
Starting from certain classes of circulant graphs, we also utilize particular constructions involving graph operations like tensor product, union, and complement to discover new classes of circulant graphs that possess four eigenvalues in their spectra.
Our motivation for this approach steams from \cite{Cv80}, in which construction of certain classes of strongly regular graphs and symmetric block-designs is presented. The constructions begin with small and simple graphs $K_2$ and $K_4$ and involve performing NEPS operations on them.
The paper \cite{Da95} provides a deeper examination of connected regular graphs with four distinct eigenvalues, presenting some properties, constructions, and examples.
Our primary source of motivation arises from \cite{ChRa18}, which introduces highly specific classes of integral circulant graphs whose spectra exhibit four distinct eigenvalues. 
These graphs possess a prime power order or an order precisely equal to the product of three distinct primes with a specific divisor set.
Additionally, a part of the study conducted in this paper focuses on characterizing strongly regular circulant graphs of prime order. This line of research was initially introduced in \cite{ChRaGu12} and subsequently expanded upon in \cite{Ba22} for integral circulant graphs.
In the cited papers \cite{ChRa18, Ba22, ChRaGu12}, there are open questions concerning the characterization of circulant graphs with three or four distinct eigenvalues in their spectra.

The plan of the paper is organized as follows.
In the paper, we begin with a preliminary section that introduces the necessary notation and concepts related to circulant graphs, graph operations, and number theoretical tools.
In Section 3 (specifically subsection 3.1), it is observed that strongly regular integral circulant graphs can be expressed as a tensor product of two simpler graphs.
This concept helps us in identifying classes of circulant and non-circulant graphs in Section 4, where a similar method is employed, resulting in spectra with four distinct eigenvalues.
We can deduce from the tensor product representation of strongly regular integral circulant graphs that their order must be composite.
Hence, in the remaining part of subsection 3.1 (more precisely in Theorem 3.3), we provide a comprehensive characterization of all strongly regular circulant graphs with prime order. 
By characterizing all iterated line graphs of unitary Cayley graphs whose spectra contain three distinct eigenvalues, we achieve a class of strongly regular graphs that do not exhibit circulant graph properties, as presented in subsection 3.2.
We initiate subsection 4.1 by presenting constructions of two classes of integral graphs whose spectra contain four distinct eigenvalues, achieved through a tensor product of a specific number of complete graphs and complete graphs with loops attached to each vertex (Theorem 4.2 and Theorem 4.3).
An example of the class of circulant graphs, as stated in Theorem 4.3, is provided in \cite{ChRa18}. This example pertains to graphs with an order that equals precisely the product of three distinct primes and a specific set of divisors.
Moreover, Theorem 4.4 introduces a class of non-integral circulant graphs whose spectra contain four distinct eigenvalues, which is obtained by employing a construction relying on the tensor product of the strongly regular circulant graph with a prime order and the complete graph with loops connected to every vertex.
Furthermore, Theorems 4.5 and 4.6 present two classes of integral circulant graphs that possess spectra with four distinct eigenvalues.
These graphs are obtained through various graph operations, including the union and the complement, primarily derived from specific subgraphs of strongly regular circulant graphs and complete graphs of certain orders.
The instance of a graph class derived in Theorem 4.5, as presented in \cite{ChRa18}, corresponds to graphs with orders that are powers of prime numbers and specific divisor sets.
In Theorem 4.7, at the conclusion of subsection 4.1, we fully characterize circulant graphs with prime orders which possess four distinct eigenvalues in their spectra, whereas all other aforementioned graph classes with four distinct eigenvalues have composite order.
Finally, in subsection 4.2, we characterize all iterated line graphs of unitary Cayley graphs whose spectra contain four distinct eigenvalues and prove that they are not circulant graphs.

The proofs in this context typically require comprehensive discussion and rely upon the interplay among (spectral) graph theory, number theory, and polynomial theory. Nevertheless, certain proof demonstrations are exclusively grounded in number theory, particularly those pertaining to the characterization of circulant graphs of prime order with three or four eigenvalues. These proofs employ various tools such as power residues, residue systems, and reciprocity.


\section{Preliminaries}

 A {\it circulant graph} $G(n;S)$ is a
graph on vertices $\Z_n=\{0,1,\ldots,n-1\}$ such that vertices $i$
and $j$ are adjacent if and only if $i-j \equiv s \pmod n$ for some
$s \in S$. Such a set $S$ is called the {\it symbol} of graph $G(n;S)$.
As we will consider undirected graphs without loops, we assume that
$S=n-S=\{n-s\ |\ s\in S\}$ and $0\not\in S$. Note that the degree of
the graph $G(n;S)$ is $|S|$. The eigenvalues and eigenvectors of
$G(n;S)$ are given by

\begin{equation} \label{eq:eigenvalues unwigted} \lambda_j=\sum_{s \in S}
\omega^{js}_n, \quad v_j=[1 \ \omega_n^j \ \omega_n^{2j} \cdots
\omega_n^{(n-1)j}]^T,
\end{equation}
 where
$\omega_n=e^{i\frac{2\pi}n}$ is the $n$-th root of unity \cite{Davis70}.

\smallskip

Circulant graphs are a subclass of the wider class of Cayley graphs.
Let $\Gamma$ be a multiplicative group with identity $e$. For
$S\subset \Gamma$, $e\not\in S$ and $S^{-1} = \{s^{-1}\ |\ s\in
S\}=S$, the Cayley graph $X = Cay(\Gamma,S)$ is the undirected graph
having vertex set $V(X)=\Gamma$ and edge set $E(X) = \{\{a,b\}\ |\
ab^{-1}\in S\}$. It is not hard to see that a graph is circulant if
it is a Cayley graph on some cyclic group, i.e. its adjacency matrix
is cyclic.

A graph is {\it integral} if all its eigenvalues are integers. A
circulant graph $G(n;S)$ is integral if and only if
$$
S=\bigcup_{d \in D} G_n(d),
$$
for some set of divisors $D \subseteq D_n$ \cite {wasin}. Here, $G_n(d)$ denotes the set of integers $k$ satisfying $\gcd(k,n)=d$ and $1\leq k \leq n-1$. The set $D_n$ encompasses all divisors of $n$ excluding $n$ itself.
Hence, an {\it integral circulant graph} is characterized by its order $n$ and the set of divisors $D$. We denote an integral circulant graph with $n$ vertices and the set of divisors $D \subseteq D_n$ as $\ICG_n(D)$. If $D={1}$, an integral circulant graph of order $n$ is denoted by $X_n$ and referred to as a unitary Cayley graph, as described in  \cite{KlSa07}.

From the
above characterization of integral circulant graphs we have that the
degree of an integral circulant graph is $\deg \ICG_n(D)=\sum_{d \in
D}\varphi(n/d). $ Here $\varphi(n)$ denotes the Euler-phi function
\cite{HardyWright}. If $D=\{d_1,\ldots,d_k\}$, it can be seen
that $\ICG_n(D)$ is connected if and only if
$\gcd(d_1,\ldots,d_k)=1$, given that $G(n;s)$ is connected if and
only if $\gcd(n, S)=1$, see \cite{Hwang03}. Moreover, the following lemma holds

\begin{lemma}
\label{lem:ICG unconnected}
 If $d_1, d_2,\ldots, d_k$ are divisors of $n$ such that the greatest common divisor $gcd(d_1, d_2,\\\ldots, d_k)$ equals $d$, then the graph $\ICG_n(d_1, d_2, \ldots, d_k)$ has exactly $d$ connected components isomorphic to $\ICG_{n/d}(\frac{d_1}{d},\ldots,\frac{d_n}{d})$.
\end{lemma}

Throughout the paper,
 we let $ p_1^{\alpha_1} p_2^{\alpha_2} \cdots
p_k^{\alpha_k}$ be the prime factorization of $n$.

\bigskip

Let us define $c(n,j)$ as follows

\begin{equation}
c(j,n)=\mu(t_{n,j})\frac{\varphi(n)}{\varphi(t_{n,j})}, \quad
t_{n,j}=\frac n{\gcd(n,j)}, \label{ramanujan} \end{equation} where
$\mu$ denotes
the M\" obius function defined as

\begin{eqnarray}
\mu(n)&=&\left\{
\begin{array}{rl}
1, &  \mbox{if}\ n=1  \\
0, & \mbox{if $n$ is not square--free} \\
(-1)^k, & \mbox {if $n$ is product of $k$ distinct prime numbers}.
\end{array} \right.
\end{eqnarray}
The expression $c(j,n)$ is known as the {\it Ramanujan function}
(\cite[p.~309]{HardyWright}). The spectrum $(\lambda_0,\ldots,\lambda_{n-1})$ of $\ICG_n(D)$  can be
expressed in terms of the Ramanujan function (see \cite{wasin}) as follows
 \begin{equation}
 \label{ldef}
 \lambda_j=\sum_{d\in D} c(j,\frac{n}{d}).
\end{equation}




 Let us observe the following
properties of the Ramanujan function. These basic and useful
properties will be extensively used throughout the remainder of the paper.

\begin{proposition} \label{prop:c} For any positive integers $n$, $j$, $d$ and prime $p$ such that $d \mid n$ and $p\mid n$,
the following are satisfied

\begin{eqnarray}
c(0,n)&=&\varphi(n), \label{prop:c0}\\
c(1,n)&=&\mu(n), \label{prop:c1}\\
c(2,n)&=&\left\{
\begin{array}{rl}
\mu(n), &  n \in 2\N+1  \\
\mu(\frac{n}{2}), & n \in 4\N+2 \\
2\mu(\frac{n}{2}), & n \in 4\N
\end{array} \right.
 \label{prop:c2}\\
c(\frac{n}{p},\frac{n}{d})&=&\left\{ \begin{array}{rl}
\varphi(\frac{n}{d}), & p \mid d \\
-\frac{\varphi(\frac{n}{d})}{p-1}, & p \nmid d \label{prop:cn/p}\\
\end{array}\right..
\end{eqnarray}
\end{proposition}

\bigskip

In the following sections, we will utilize additional tools from number theory, specifically from the theory of quadratic and cubic residues, to prove certain theorems. For a given prime number $p$ and an integer $a$, the {\it Legendre symbol} is defined as follows

\begin{eqnarray*}
\big{(}\frac{a}{p}\big{)}&=&\left\{
\begin{array}{rl}
1, &  \mbox{if $p\nmid a$ and a is quadratic residue modulo $p$};  \\
-1, & \mbox{if $p\nmid a$ and a is quadratic non-residue modulo $p$}; \\
0, & \mbox{if $p\mid a$}.
\end{array} \right.
 \label{prop:c2}\\
\end{eqnarray*}

Euler's criterion and the property of the multiplicity of the Legendre symbol will be two crucial facts that will be frequently used. 
Indeed, these two properties can be formulated as follows

\begin{eqnarray*}
&&\big{(}\frac{a}{p}\big{)}\equiv_{p} a^{\frac{p-1}{2}} \mbox{ (Euler's criterion)},\\
&&\big{(}\frac{ab}{p}\big{)}=\big{(}\frac{a}{p}\big{)}\big{(}\frac{b}{p}\big{)}=1, \mbox{for integers $a$, $b$ and prime $p>2$}.\\
\end{eqnarray*}

The following theorem represents a generalization of Euler's criterion and will be used in the theory of cubic residues. For more about the theory of quadratic and cubic residues one may refer to \cite{DeDe95,HardyWright}.

\begin{theorem}
\label{thm:residues}
    $x^k\equiv_{p} a$ has a solution if and only if $a^{\frac{p-1}{d}}\equiv_{p} 1$, where $d=\gcd(k,p-1)$. If the congruence has a solution, then it actually has $d$ incongruent solutions modulo $p$.
\end{theorem}

For a given prime number $p$ and an integer $a$, the {\it  the rational cubic residue symbol} is defined as follows:

\begin{eqnarray*}
\big{[}\frac{a}{p}\big{]}_{3}&=&\left\{
\begin{array}{rl}
1, &  \mbox{if $p\nmid a$ and a is quadratic residue modulo $p$};  \\
-1, & \mbox{if $p\nmid a$ and a is quadratic non-residue modulo $p$}; \\
0, & \mbox{if $p\mid a$}.
\end{array} \right.
 \label{prop:c2}\\
\end{eqnarray*}

\bigskip

Let us remind that the spectral radius of a connected $r-$regular graph $X$ is equal to the regularity $r$ and it is a simple eigenvalue of $X$.
According to (\ref{prop:c0}), in the case of an integral circulant graph with the spectrum $(\lambda_0,\ldots,\lambda_{n-1})$ given by (\ref{ldef}), $\lambda_0$ is equal to the regularity of the graph.


 The well-known characterization of strongly
regular graphs will be used in the paper.
\begin{lemma}
\cite{GoRo01}
\label{lem:srg_eig}
A connected regular graph is strongly regular if and only if it has exactly three
distinct eigenvalues.
\end{lemma}
If we denote the eigenvalues of a strongly regular graph as $r$ (representing the regularity of the graph), $\theta$ (with multiplicity $m_{\theta}$), and $\tau$ (with multiplicity $m_{\tau}$), then the following equalities can be observed (referring to equation (10.2) in \cite{GoRo01})
\begin{eqnarray}
\label{parametri-sopstvene vrednosti} \theta+\tau &=&a-c,\quad\quad
\theta\tau=c-r\\
m_{\theta}&=&-\frac{(n-1)\tau+r}{\theta-\tau}, \quad
m_{\tau}=\frac{(n-1)\theta+r}{\theta-\tau}. \label{visestrukosti}
\end{eqnarray}

\bigskip

A {\it block-design} is a collection $\Lambda$ of $b$ subsets (blocks) of a set of points $S={x_1, x_2,\ldots, x_v}$. The block design satisfies the following conditions: each subset contains $k$ elements, and every pair of elements from $S$ appears in $\lambda$ subsets.



A block design is considered {\it symmetric} when the number of blocks is equal to the number of points, i.e., $b=v$. 
A block-design with $b=v$ is called {\it symmetric}. 
Consequently, we represent a symmetric block design using the triple of parameters $(v,k,\lambda)$.

The {\it incidence graph} $X$ of a block-design is the graph with
vertex set $S\cup \Lambda$, where two vertices $x\in S$ and
$B\in\Lambda$ are adjacent if $x\in B$. The incidence graph of a
block-design is bipartite graph with four distinct eigenvalues.

\bigskip

If $G$ is a graph, then the {\it line graph} $L(G)$ of $G$ s constructed by considering the edges of $G$ as vertices in $L(G)$, any two of them being adjacent
if the corresponding edges of $G$ have a vertex of $G$ in common.

\bigskip

Finally we give a definition of the {\it tensor product} of two
graphs. The tensor product $G\otimes H$ of graphs $G$ and $H$ is a
graph such that the vertex set of $G \otimes H$ is the Cartesian
product $V(G) \times V(H)$ and any two vertices $(u,u')$ and
$(v,v')$ are adjacent in $G \otimes H$ if and only if $u'$ is
adjacent with $v'$ and $u$ is adjacent with $v$.

\section{Graph matrices with three distinct eigenvalues}

\subsection {Construction of strongly regular graphs using tensor product}

We  find strongly regular
circulant graphs by starting from graphs with small number of distinct
eigenvalues and performing some graph operations such as tensor
product. Also, it is well-known that if $\lambda_1, \ldots,
\lambda_n$ are eigenvalues of the adjacency matrix of a graph $G$
and $\mu_1,\ldots , \mu_m$ are eigenvalues of the adjacency matrix
of a graph $H$, then the eigenvalues of the adjacency matrix of the
tensor product $G\otimes H$ are $\lambda_i \cdot \mu_j$ for $1\leq
i\leq n$ and $1\leq j\leq m$.
Therefore, for some composite number $n$ and arbitrary divisor
$d\mid n$, $1<d<n$, we can start with complete graphs $K_d$ and $K_{\frac{n}{d}}^*$,
where every vertex of $K_{\frac{n}{d}}^*$ has a loop, and perform the tensor
operation in the following way
\begin{eqnarray}
\left.
\begin{array}{l}
Sp(K_d)=\{{d-1}^{(1)}, -1^{(d-1)}\}\\
Sp(K^*_{\frac n d})=\{{\frac n d}^{(1)}, 0^{(\frac
n d-1)}\}\\
\end{array}\right\}\Rightarrow
\left.
\begin{array}{l}
Sp(K_d\otimes K^*_{\frac n d})=\{(d-1)\frac
{n}{d}^{(1)}, 0^{(n-d)}, -\frac{n}{d}^{(d-1)}\}.\label{kronecker_complete} \\
\end{array}\right.
\end{eqnarray}

The notation $Sp(G)$ denotes the spectrum of graph $G$, which encompasses the eigenvalues along with their respective multiplicities. It is evident that the graph $K_d\otimes K^*_{\frac n d}$ is regular, exhibiting a regularity of $\frac{(d-1)n}{d}$. Additionally, this graph possesses precisely three distinct eigenvalues, thereby establishing it is a strongly regular graph. Moreover, according to
(\ref{parametri-sopstvene vrednosti}) its parameters are
$r=c=\frac{(d-1)n}{d}$ and $a=\frac{(d-2)n}{d}$. Due to the fact that  any two nonadjacent vertices
have the same neighbourhood of the size $\frac{(d-1)n}{d}$, it can be concluded these two
vertices actually belong to the independent set of the size
$n-\frac{(d-1)n}{d}=\frac{n}{d}$. This means that $K_d\otimes K^*_{\frac n d}$ is
isomorphic to the complete multipartite  graph
$K_{\underbrace{n/d,\ldots,n/d}_{d}}$, which exhibits a circulant structure. Let us note that this graph coincides with the strongly regular graph $\ICG_n(d'\in D_n\ |\ d\nmid d')$ as derived from  Theorem 15 in \cite{Ba22}.
However, the idea of exploiting the tensor product operation on graphs that initially have two or three distinct eigenvalues in their spectra will be utilized extensively in the following section. This approach aims to construct graphs that exhibit four distinct eigenvalues in their spectra.
Moreover, it is evident that the graph resulting from the tensor product $K_d\otimes K^*_{\frac{n}{d}}$ has a composite order.
Therefore, in the remaining part of this section, we will introduce a class of strongly regular circulant graphs with non-integer spectra of prime order. 
Furthermore, it turns out that this is the only class of strongly regular circulant graphs of prime order, which will be demonstrated in the following theorem.
This class of graphs will be exploited in discovering new classes of circulant graphs that possess four distinct eigenvalues in their spectra.
We use the following well-known lemmas in the proof of the theorem. 

\begin{lemma}
\label{lem:polynomial} Let $p$ be an arbitrary prime number and
$P(x)\in \Z[x]$ a polynomial of degree at most $p-1$ having
$\omega_p$ as a root. Then $P(x)=c(x^{p-1}+x^{p-2}+...+x+1)$ where
$c\neq 0$ is an integer.
\end{lemma}

\begin{lemma}
\label{lem:eigenvalues equality} 
Let $G(p;S)$ be a circulant graph with prime order $p$ and $\lambda_0,\ldots,\lambda_{p-1}$ its spectrum given by the equation (\ref{eq:eigenvalues unwigted}). Then, it holds that 
\begin{equation*}
\label{eq:eigenvalues equality} \lambda_i=\lambda_j \
\Longleftrightarrow \{r_{i,s}|\ s\in S\}=\{r_{j,s}|\ s\in S\},
\end{equation*}
for $1\leq i,j\leq p-1$, and $r_{i,s}$ represents the residue of $is$ modulo $p$ for $1\leq i\leq p-1$ and $s\in S$ .
\end{lemma}
\begin{proof}
It can be observed from (\ref{eq:eigenvalues unwigted}) that $\lambda_i=\lambda_j \Longleftrightarrow \sum_{s\in S}\omega_p^{r_{i,s}}-\sum_{s\in S}\omega_p^{r_{j,s}}=0$. This implies that $\omega_p$ is a root of the polynomial $P(x)=\sum_{s\in S}x^{r_{i,s}}-\sum_{s\in S}x^{r_{j,s}}\in \mathbb{Z}[x]$. Considering that the polynomial $P(x)$ has a degree of at most $p-1$, as stated in Lemma \ref{lem:polynomial}, it can be either equal to $cA(x)$ where $A(x)=x^{p-1}+x^{p-2}+...+x+1$, or it can be the zero polynomial. However, since $P(1)=0$ and $A(1)=cp\neq 0$, it follows that $P(x)\neq cA(x)$, and therefore $P(x)$ must be the zero polynomial.
This implies that  $\lambda_i=\lambda_j \
\Longleftrightarrow \{r_{i,s}|\ s\in S\}=\{r_{j,s}|\ s\in S\}$.

    \qed
\end{proof}

\begin{theorem} 
\label{thm:strongly-regular prime}
For a prime number $p$ the circulant graph $G(p;S)$  is strongly regular if and only if $S$ is a set of all quadratic residues modulo $p$
 or all  quadratic non-residues modulo $p$ and $p\in 4\N+1$.
\end{theorem}
\begin{proof}
  
Let $G(p;S)$ be a strongly regular graph. Considering the regularity of $G(p;S)$, which is equal to $|S|$, and the fact that $\lambda_0=|S|$, it follows that the sequence of eigenvalues $\lambda_1,\lambda_2,...,\lambda_{p-1}$, given by (\ref{eq:eigenvalues unwigted}), must consist of exactly two distinct eigenvalues (according to Lemma \ref{lem:srg_eig}).

Denote by $r_{i,s}$ the residue of $is$ modulo $p$, for $1\leq i\leq p-1$ and $s\in S$. 
If we denote $S_i=\{r_{i,s}|\ s\in S\}$, for any $1\leq i\leq p-1$, according to Lemma \ref{lem:eigenvalues equality},   we conclude that 

\begin{equation}
\label{eq:eigenvalues equality1} \lambda_i=\lambda_j \
\Longleftrightarrow S_i=S_j.
\end{equation}
Therefore, since $S_1=S$, the number of distinct eigenvalues in the spectrum of $G(p;S)$ is equal to three if and
only if $\{S_i|\ 1\leq i\leq p-1\}=\{S,T\}$ for some $T,\
T\subseteq \{1,\ldots,p-1\}$.
Moreover, for two distinct
integers $s_1,s_2\in S$ we have $r_{i,s_1}\neq r_{i,s_2}$, which
implies $|S_i|=|\{r_{i,s}|\ s\in S\}|=|S|=|T|$, for $1\leq i\leq p-1$. 
Furthermore, for a given $s\in S$, we can conclude that $\{r_{i,s}|\ 1\leq i\leq p-1\}=\{1,\ldots,p-1\}$, since $\{1,\ldots,p-1\}$ forms a reduced residue system modulo $p$ and $\gcd(s,p)=1$. Therefore, we have $S\cup T=\{1,\ldots,p-1\}$.




%


\smallskip

Now, we will show that $S$ and $T$ are disjoint.
Suppose that there exists some $c\in S\cap T$. This means that $c\in S_i$ for $1\leq i\leq p-1$. Therefore, for every $1\leq i\leq p-1$, there exists $s\in S$ such that $c=r_{i,s}$, and consequently, $s\equiv_p c\cdot i^{-1}$, where $i^{-1}$ denotes the modular inverse of $i$ modulo $p$. This implies that $\{c\cdot i^{-1}\ |\ 1\leq i\leq p-1\}\subseteq S$.
On the other hand, since $p$ is a prime number, both sets $\{i^{-1}|\ 1\leq i\leq p-1\}$ and $\{c\cdot i^{-1}\equiv s\ |\ 1\leq i\leq p-1\}$ form reduced residue systems modulo $p$. Therefore, we can conclude that $\{1,\ldots,p-1\}\subseteq S$.
Given that $|S|=|T|$, we finally obtain $S=T=\{1,\ldots,p-1\}$. However, this contradicts the fact that $S\neq T$, and we have proved that $S\cap T=\emptyset$.

Since $|S|=|T|$, $S\cup T=\{1,\ldots,p-1\}$ and $S\cap T=\emptyset$, it  holds that $|S|=|T|=\frac{p-1}{2}$. According to
Lemma \ref{eq:eigenvalues equality}, $\lambda_i=\lambda_j$ yields
$\prod_{s\in S}r_{i,s}=\prod_{s\in S}r_{j,s}$ and thus $\prod_{s\in
S}is\equiv_p\prod_{s\in S}js$ and $i^{|S|}\equiv_pj^{|S|}$. Using
Euler's Criterion  we obtain $\big{(}\frac{i}{p}\big{)}\equiv_p
i^{\frac{p-1}{2}}\equiv_p
j^{\frac{p-1}{2}}\equiv_p\big{(}\frac{j}{p}\big{)}$. This means that
if $\lambda_i=\lambda_j$, then the numbers $i$ and $j$ are either
both quadratic residues or both quadratic non-residues modulo $p$.
Since $S_1=S$ it follows $\{i|\ S_i=S,\
1\leq i\leq p-1\}\subseteq \{i|\
\big{(}\frac{i}{p}\big{)}=1,\
1\leq i\leq p-1\}$. Similarly, we conclude that $\{i|\ S_i=T,\
1\leq i\leq p-1\}\subseteq \{i|\
\big{(}\frac{i}{p}\big{)}=-1,\
1\leq i\leq p-1\}$. Given that $\{i|\
\big{(}\frac{i}{p}\big{)}=1,\
1\leq i\leq p-1\}=|\{i|\
\big{(}\frac{i}{p}\big{)}=-1,\
1\leq i\leq p-1\}|=\frac{p-1}{2}$ and $\{i|\ S_i=S,\
1\leq i\leq p-1\}\cup\{i|\ S_i=T,\
1\leq i\leq p-1\}=\{1,\ldots, p-1\}$, we have that $\{i|\ S_i=S,\
1\leq i\leq p-1\}= \{i|\
\big{(}\frac{i}{p}\big{)}=1,\
1\leq i\leq p-1\}$ and $\{i|\ S_i=T,\
1\leq i\leq p-1\}= \{i|\
\big{(}\frac{i}{p}\big{)}=-1,\
1\leq i\leq p-1\}$ 

Suppose there exists $s\in S$ such that $\big{(}\frac{s}{p}\big{)}=1$. For every $1\leq i\leq p-1$ such that $\big{(}\frac{i}{p}\big{)}=1$, we have that 
$S_i=S$ and $r_{i,s}\in S_i$, which implies that $r_{i,s}\in S$. Since 
$\big{(}\frac{is}{p}\big{)}=\big{(}\frac{i}{p}\big{)}\big{(}\frac{s}{p}\big{)}=1$, we conclude that $\{r_{i,s}\ |\ \big{(}\frac{i}{p}\big{)}=1,\  1\leq i\leq p-1 \}\subseteq \{i|\
\big{(}\frac{i}{p}\big{)}=1,\
1\leq i\leq p-1\}$. Moreover, from the fact that $i\neq j$ implies $r_{i,s}\neq r_{j,s}$, for $1\leq i,j\leq p-1$,
we further get that $\{r_{i,s}\ |\ \big{(}\frac{i}{p}\big{)}=1,\  1\leq i\leq p-1\}= \{i|\
\big{(}\frac{i}{p}\big{)}=1,\
1\leq i\leq p-1\}$. Finally, from the preceding discussion it can be concluded that 
$\{i|\ \big{(}\frac{i}{p}\big{)}=1,\ 1\leq i\leq p-1\}\subseteq S$ and since $|\{i|\
\big{(}\frac{i}{p}\big{)}=1,\
1\leq i\leq p-1\}|=|S|=\frac{p-1}{2}$, it holds that $\{i|\
\big{(}\frac{i}{p}\big{)}=1,\
1\leq i\leq p-1\}= S$. If we assume that there exists $s\in S$ such that $\big{(}\frac{s}{p}\big{)}=-1$, it can be proven in a similar fashion $\{i|\
\big{(}\frac{i}{p}\big{)}=-1,\
1\leq i\leq p-1\}= S$.



Now, we will prove that $p \in 4\mathbb{N}+1$. Without loss of generality, let us assume that $S$ is the set of all quadratic residues. According to the definition of $S$ as $S = p - S$, we can conclude that both $1$ and $p-1$ are elements of $S$. This implies that $\left(\frac{-1}{p}\right) = 1$. By applying Euler's Criterion, we further deduce that $(-1)^{\frac{p-1}{2}} \equiv_p 1$, and it follows that $\frac{p-1}{2}$ is even, as we set out to prove.

\medskip

Suppose now that $S$ is the set of all quadratic
residues and $p\in 4\N+1$. If $x\in S,$ then for any $s\in S$ we have
$\big{(}\frac{xs}{p}\big{)}=\big{(}\frac{x}{p}\big{)}\big{(}\frac{s}{p}\big{)}=1$
and $r_{x,s}\in S$ which implies $S_x\subseteq S$. Since
$|S_x|=|S|=\frac{p-1}{2}$ we obtain $S_x=S$. If $T$ is the set of all quadratic non-residues and $x\in T$, then for
every $s\in S$ there holds
$\big{(}\frac{xs}{p}\big{)}=\big{(}\frac{x}{p}\big{)}\big{(}\frac{s}{p}\big{)}=-1$
which yields $r_{x,s}\in T$. Thus, we have $S_x\subseteq T$ and
$S_x=T$. 
This way we have proved that $G(p;S)$ has exactly three distinct eigenvalues in its spectrum. Similar conclusion can be derive for $S$ being set of all quadratic non-residues.

In the rest of the proof, we show that $S=p-S$. It is sufficient to prove that  $-1$ is a quadratic residue modulo $p$.
Using Euler's Criterion, we have that $\left(\frac{-1}{p}\right)\equiv_{p} (-1)^{\frac{p-1}{2}} \equiv_p 1$, which is supposed to be proven.

\qed
\end{proof}

It can be easily concluded that a strongly regular graph $G(p;S)$, for $p\in 4\N+1$, is a self-complementary graph. Indeed, if $S$ contains all quadratic residues modulo $p$, then the set of symbols of $\overline{G(p;S)}$ contains all quadratic non-residues. We can establish a bijection $f:\{0,1,\ldots,p-1\}\rightarrow \{0,1,\ldots,p-1\}$ such that $f(x)=r_{b,x}$ for all 
$x\in \{0,1,\ldots,p-1\}$ and some non-quadratic residue $b$ modulo $p$. It can be easily shown that this bijection is an isomorphism. 

Now, we can proceed with determining of the spectrum and parameters of the strongly regular graph $G(p;S)$, where $p\in 4\N+1$ and $S$ is the set of quadratic residues modulo $p$. 
Let $\lambda_0,\ldots,\lambda_{p-1}$ be the spectrum of $G(p;S)$ given by (\ref{eq:eigenvalues unwigted}).
From the proof of Theorem \ref{thm:strongly-regular prime}, we immediately see that the regularity of the graph $r$ is equal to the eigenvalue $\lambda_0=|S|=\frac{p-1}{2}$. Now, let $j$ be a quadratic residue modulo $p$. Since the equation $x^2\equiv_{p} i$ has two incongruent solutions modulo $p$, according to Theorem  \ref{thm:residues}, we see that $1+2\lambda_j=\sum_{i=0}^{p-1}\omega_p^{i^2}$.  
Let $z$ denotes the sum $z=\sum_{i=0}^{p-1}\omega_p^{i^2}$. We see that $|z|^2=z\overline{z}=(\sum_{i=0}^{p-1}\omega_p^{i^2}) (\sum_{i=0}^{p-1}\omega_p^{-i^2})=\sum_{i=0}^{p-1}\sum_{j=0}^{p-1}\omega_p^{i^2-j^2}=\sum_{i=0}^{p-1}\sum_{j=0}^{p-1}\omega_p^{(i-j)(i+j)}$. Given that we can establish a bijection $g:\{0,1,\ldots,p-1\}\times \{0,1,\ldots,p-1\}\rightarrow \{0,1,\ldots,p-1\}\times \{0,1,\ldots,p-1\}$ such that $g(i,j)=(r_{1,i-j}, r_{1,i+j})$, then the last sum $\sum_{i=0}^{p-1}\sum_{j=0}^{p-1}\omega_p^{(i-j)(i+j)}$ is equal to $\sum_{i=0}^{p-1}\sum_{j=0}^{p-1}\omega_p^{ij}$. Finally, we have that $\sum_{i=0}^{p-1}\sum_{j=0}^{p-1}\omega_p^{ij}=\sum_{j=0}^{p-1}\omega_p^{0}+\sum_{i=1}^{p-1}\frac{\omega_p^{ip}-1}{\omega_p^i-1}=p$, and therefore $|z|=\sqrt{p}$. As $z\in \R$, suppose that $z=\sqrt{p}$.  This directly implies that $\lambda_i=\frac{z-1}{2}=\frac{\sqrt{p}-1}{2}$. Let $\tau=\lambda_i$ and $\theta$ be the negative eigenvalue in the spectrum of $G(p;S)$. According to the proof of the preceding theorem, we have that $m_\theta=m_\tau=\frac{p-1}{2}$.  From (\ref{visestrukosti}), we deduce that $\theta=-\frac{\tau(p+1-m_{\theta})+r}{m_{\theta}}$. Moreover, since $p+1-m_{\theta}=r=m_{\theta}=\frac{p-1}{2}$, we have that $\theta=-\tau-1=-\frac{\sqrt{p}-1}{2}-1=-\frac{\sqrt{p}+1}{2}$. Furthermore, using (\ref{parametri-sopstvene vrednosti}), we see that $c=\theta\tau+r$, which directly implies that $c=\frac{p-1}{4}$. 
Taking into account the same equation, it can be observed that $a=\theta+\tau+c$, and a simple calculation leads to the result $a=\frac{p-5}{4}$. If we assume $z=-\sqrt{p}$, we can determine that $\tau=-\frac{\sqrt{p}+1}{2}$ and $\theta=\frac{\sqrt{p}-1}{2}$. This indicates that we obtain the same set of eigenvalues in this case.

\subsection {Construction of strongly regular graphs using line operator}

In this section, we explore additional classes of strongly regular graphs by using the line graph operator $L$ on the class of unitary Cayley graphs. This concept arises from the well-known observation that two classes of strongly regular graphs can be obtained by applying the line operator $L$ to complete graphs $K_n$ and complete bipartite graphs $K_{n,n}$. It is worth mentioning that both of these classes are circulant graphs, similar to unitary Cayley graphs. The graph $L(K_n)$ has the parameters $(\frac{n(n-1)}{2}, 2n-4, n-2, 4)$ (also known as triangular graphs), while $L(K_{n,n})$ has parameters $(n^2, 2n-2, n-2, 2)$ (referred to as square lattice graphs). We will demonstrate that by repeatedly applying the line operator, we can obtain strongly regular graphs that are not circulant graphs

To begin with, we will describe all strongly regular graphs within the class of unitary Cayley graphs.

\begin{theorem}
\label{thm:sruc} Unitary Cayley graph $X_n$ is strongly regular if
and only if $n$ is composite prime power.
\end{theorem}
\begin{proof}
Recall that a connected regular graph is strongly regular if it has
exactly three distinct eigenvalues (according to Lemma \ref{lem:srg_eig}). Since $X_n$ is a connected
regular graph with spectral radius $\lambda_0$, it is sufficient to
characterize all Unitary Cayley graphs such that the set
$\{\lambda_1,\ldots,\lambda_{n-1}\}$ contains exactly two distinct
values.

Suppose that the set $\{\lambda_1,\ldots,\lambda_{n-1}\}$ contains
exactly two distinct values. We distinguish two cases depending on the
different values of $k$.

\medskip

{\bf Case 1.} $k=1$. For $n=p_1$, it is clear that $X_{p_1}$ is complete and it is not strongly regular by definition.
Let $n={p_1}^{\alpha_1}$ for $\alpha_1\geq 2$ and let $1\leq j\leq n-1$
be an arbitrary index such that $j={p_1}^\beta M$, for $0\leq \beta\leq
\alpha_1-1$ and $M\in 2\N+1$. Then, we conclude that
$t_{n,j}={p_1}^{\alpha_1}/\gcd({p_1}^{\alpha_1},{p_1}^{\beta}M)={p_1}^{\alpha_1-\beta}$.
Furthermore, as
\begin{eqnarray}
\label{sopstvenese za jako regularni unitarni}
\lambda_j=c(j,n)&=&\mu({p_1}^{\alpha_1-\beta})\varphi({p_1}^{\alpha_1})/\varphi({p_1}^{\alpha_1-\beta})=\left\{
\begin{array}{rl}
0, & 0\leq \beta \leq \alpha_1-2 \\
-{p_1}^{\alpha_1-1}, & \beta=\alpha_1-1 \\
\end{array}\right.,
\end{eqnarray}
it is clear that $X_{p_1^{\alpha_1}}$ is a strongly regular graph, for
$\alpha_1\geq 2$.

\medskip

{\bf Case 2.} $k\geq 2$. Let $p$ be an arbitrary prime
divisor of $n$. Considering (\ref{prop:cn/p}), we
obtain that
$\lambda_{n/p}=c(n/p,n)=-\varphi(n)/(p-1)$.
Since $k\geq 2$, then there exists at least two prime divisors $p_i$
and $p_j$ of $n$ such that $i<j$ and
$|\lambda_{n/p_i}|>|\lambda_{n/p_j}|$. Furthermore, as $\lambda_1=\mu(n)$, according to (\ref{prop:c1}), we see that
$|\lambda_{n/p_i}|>|\lambda_{n/p_j}|\geq  1\geq \lambda_1$ and hence $X_n$ is not strongly regular, if $|\lambda_{n/p_j}|>1$.
On the other hand, the equality $|\lambda_{n/p_j}|=  1$ holds if and only if  $n=2p_2$. However, in this case 
$\lambda_{n/p_j}=  -1$, $\lambda_1=\mu(n)=1$, and hence $\lambda_1\neq \lambda_{n/p_j}$. It follows that the graph $X_{2p_2}$ does not meet the criteria for being strongly regular.

 \qed
\end{proof}

It is easy to see that strongly regular graphs in the class of unitary Cayley graphs are
complete multipartite graphs on $p^\alpha$ ($\alpha\geq 2$) vertices, with regularity $p^{\alpha-1}(p-1)$. According to
(\ref{parametri-sopstvene vrednosti}), any two distinct adjacent
vertices have $p^{\alpha-1}(p-2)$ common neigbours, while any to
nonadjacent vertices have $p^{\alpha-1}(p-1)$ common neighbours. Furthermore, the multiplicities  of the eigenvalues $\theta$ and $\tau$ are $m_{\theta}=p(p^{\alpha-1}-1)$ and $m_{\tau}=p-1$, respectively.

\bigskip

We use the following theorem due to the fact that the line graph of a regular graph is also a regular graph.

\begin{theorem}
\cite{Sa67} Let $G$ be $k$ regular connected graph with $n$ vertices
and $\lambda_1,\ldots, \lambda_n$ the eigenvalues of its adjacency matrix. Then the spectrum
of $L(G)$ consists of $-2$ with multiplicity $\frac{kn}{2} - n$ and
$k + \lambda_i - 2$ for every $1 \leq i \leq n$.
\label{lg-regular}
\end{theorem}

This theorem enables us to easily observe that the
number of distinct eigenvalues of a regular $G$ is less than or equal 
 to the number of distinct eigenvalues of $L(G)$.

\begin{theorem}
\label{thm:sr-line of ucg}
Let $X_n$ be unitary Cayley graph of the order $n$. Then the line graph
of $X_n$ is strongly regular if and only if $n$ is either prime
greater than $3$ or $n$ is a power of $2$ greater than $2$.
\end{theorem}
\begin{proof}
Let $k=\varphi(n)$ be the regularity of $X_n$.

If $kn/2-n=0$, then the number of distinct eigenvalues of $L(X_n)$ is
equal to the number of distinct eigenvalues of $X_n$. In this case, we have that
$k=\varphi(n)=2$, and therefore $n\in\{3,4,6\}$. As $X_n\simeq C_n \simeq
L(X_n)$ for $n\in\{3,4,6\}$, according to Theorem \ref{thm:sruc}, we
have that $L(X_4)$ is the only strongly regular graph.

If $kn/2-n>0$, then one eigenvalue of $L(X_n)$ must be $-2$. Now,
suppose that $L(X_n)$ is strongly regular. As $L(X_n)$ has three
distinct eigenvalues then either $X_n$ has three distinct
eigenvalues such that one of them is equal to $-k$
($k+\lambda_i-2=-2$) or $X_n$ has two distinct eigenvalues and
none of them is equal to $-k$.

According to Theorem \ref{thm:sruc}, $X_n$ has three distinct
eigenvalues if and only if $n=p^\alpha$, for some prime $p$ and $\alpha\geq 2$,
and using (\ref{sopstvenese za jako regularni unitarni}) the eigenvalues of $X_n$ are $\{p^{\alpha-1}(p-1), 0,
-p^{\alpha-1}\}$. Furthermore, according to Theorem \ref{lg-regular}, any eigenvalue of $L(X_n)$ takes one of the
following values $\{2p^{\alpha-1}(p-1)-2, p^{\alpha-1}(p-1)-2,
p^{\alpha-1}(p-2)-2, -2 \}$. Therefore, as $p^{\alpha-1}(p-2)-2$ is the second smallest eigenvalue, we deduce, in this case,
that $L(X_n)$ is strongly regular if and only if $n=p^{\alpha}$ and
$p^{\alpha-1}(p-2)-2=-2$, that is, for $n=2^{\alpha}$.

If $X_n$ has exactly two eigenvalues it is a complete graph and it is easy to see that $n=p$,
for some prime $p$, and the eigenvalues of $X_n$ are $\{p-1, -1\}$.
It is clear that none of them is equal to $-k=-\varphi(p)=-p+1$, for
$p>2$, and therefore we conclude that $L(X_p)$ is strongly regular,
for $p>3$, and the eigenvalues of $L(X_p)$ are $\{2(p-1)-2,p-4,-2\}$. \qed
\end{proof}

Using the following statement, we will prove that the founded classes $L(X_p)$ for prime $p>3$,
and $L(X_{2^{\alpha}})$ for $\alpha\geq 2$, establish new classes
of strongly regular graphs, by proving that are not circulant graphs
(with the exception of $X_4\simeq C_4$).

\begin{theorem}
\cite{BrHo12} \label{line graphs and circulants} Let $G$ be a
connected graph such that $L(G)$ is a circulant. Then $G$ must
either be $C_n$, $K_4$, or $K_{a,b}$ for some $a$ and $b$ such that
$\gcd(a,b) = 1$.
\end{theorem}


Suppose first that $L(X_p)$ for prime $p>3$, is circulant. Since
$X_p$ is isomorphic to the complete graph $K_p$, for prime $p>3$,
then it cannot be isomorphic to any of the following
graphs $C_n$, $K_4$, or $K_{a,b}$ for some $a$ and $b$ such that
$\gcd(a,b) = 1$.

Now, suppose that $L(X_{2^{\alpha}})$ for $\alpha\geq 2$, is circulant.
It is easy to see that $X_{2^{\alpha}}$ is a complete bipartite graph
with the independent sets $C_i=\{0\leq j\leq 2^{\alpha}-1\ |\
j\equiv_2 i\}$, for $i\in\{0,1\}$. Therefore, $X_{2^{\alpha}}\simeq
K_{2^{\alpha-1},2^{\alpha-1}}$ and since $\alpha\geq 2$ we have
$\gcd(2^{\alpha-1},2^{\alpha-1})\neq 1$. 
Furthermore, the number of edges in $X_{2^{\alpha}}$ is equal to $\frac{2^{\alpha}\cdot2^{\alpha-1}}{2}$, which is distinct from the number of edges in $C_{2^{\alpha}}$, that is $2^{\alpha}$ (for $\alpha>2$). Therefore, based on Theorem \ref{line graphs and circulants}, we can conclude that strongly regular line graphs of unitary Cayley graphs are not circulants.

\bigskip

For $k\geq 1$, the $k$-th iterated line graph of $G$ is $L^k(G) =
L(L^{k-1}(G))$, where $L^0(G) = G$ and $L^1(G) = L(G)$. In the
following theorem we prove that the class of strongly regular graphs
derived from unitary Cayley graphs can not be extended any more
using the line graph operation.

\begin{theorem}
\label{thm:L^2-sr}
Let $X_n$ be unitary Cayley graph of the order $n$. Then $L^2(X_n)$ is
strongly regular if and only if $n=4$.
\end{theorem}
\begin{proof}
 Suppose that
$L^2(X_n)$ is connected strongly regular, i.e., $L^2(X_n)$ has three
distinct eigenvalues. This means that $L(X_n)$ may have exactly
three distinct eigenvalues, since $L(X_n)$ can not be complete. Indeed, since the order of $L(X_n)$ is equal to $\frac{n\varphi(n)}{2}$, the regularity of $L(X_n)$ is equal to $2(\varphi(n)-1)$ and the relation $\frac{n\varphi(n)}{2}-1=2(\varphi(n)-1)$ is never satisfied for $n\neq 3$, we have that $L(X_n)$ is not a complete graph.
Thus, we assume that $L(X_n)$ has three distinct eigenvalues, and according Theorem \ref{thm:sr-line of ucg} we
distinguish two cases depending on the values of $n$.

Suppose that $n=p$, for some prime $p>3$. By Theorem \ref{thm:sr-line of ucg}, the distinct eigenvalues
of $L(X_p)$ are  $\{2(p-1)-2, -2,p-4\}$. Moreover, from Theorem
\ref{lg-regular}, we obtain that the regularity of $L^2(X_p)$ is
equal to $2(2(p-1)-2)-2$ and for arbitrary eigenvalue $\lambda_i$ of
$L^2(X_p)$ holds that\\ $\lambda_i\in \{4p-10, -2,
2(p-1)-6,3p-10\}$. Any two values from this set are mutually
distinct, since $p>3$, whence we conclude that $L^2(X_p)$ is not
strongly regular.

Now, suppose that $n=2^{\alpha}$, for $\alpha\geq 2$. According to (\ref{sopstvenese za jako regularni unitarni})
we see that the distinct eigenvalues of $L(X_n)$ are $\{2^{\alpha-1}, 0, -2^{\alpha-1}\}$ and from
Theorem \ref{lg-regular}  the distinct eigenvalues of $L(X_n)$ are
$\{2^{\alpha}-2, 2^{\alpha-1}-2, -2 \}$. Therefore, the possible values for the eigenvalues of $L^2(X_n)$ are $\{2^{\alpha+1}-6, 3(2^{\alpha-1}-2),
2^{\alpha}-6, -2 \}$. Finally, we conclude that $L^2(X_n)$ has three
distinct eigenvalues if $2^{\alpha}-6=-2$, that is $\alpha=2$, as $2^{\alpha+1}-6> 3(2^{\alpha-1}-2)>
2^{\alpha}-6\geq  -2$.\qed
\end{proof}

Since $X_4\simeq C_4$, using the line operation any further will not result in any additional classes of strongly regular graphs beyond those that have already been found.

\bigskip

\section{Graph matrices with four eigenvalues}

\subsection {Construction of regular graphs with four eigenvalues using graph operations}

Building upon the concept introduced in the preceding section, we can generate regular graphs with four distinct eigenvalues by employing various graph operations, including tensor product, union, and complement. These operations are applied to graphs (connected or disconnected) whose spectra already possess two or three different eigenvalues. Indeed, for a given composite numbers $n$ and $m$, and arbitrary
divisor such that $d\mid n$, $d\mid m$ and $\min\{n,m\}>d>1$, we can use the spectrum of the graph $K_d\otimes
K^*_{\frac n d}$, that is, $Sp(K_d\otimes
K^*_{\frac n d})=\{(d-1)\frac {n}{d}^{(1)}, 0^{(n-d)},
-\frac{n}{d}^{(d-1)}\}$, and perform the operation of the tensor product in the
following way

\begin{eqnarray}
\label{spectra tensor}
&& Sp((K_d\otimes K^*_{\frac n d})\otimes (K_d\otimes
K^*_{\frac m
d}))=\nonumber\\
&&\{(d-1)^2\frac{mn}{d^2}^{(1)},0^{(nm-d^2)},-(d-1)\frac{mn}{d^2}^{(2(d-1))},\frac{mn}{d^2}^{((d-1)^2)}\}.
\end{eqnarray}

Using this construction, we obtain new classes of  regular graphs with four different
eigenvalues with the composite order $nm$ and $d>2$. Based on a computer search for constructed graphs with up to $10000$ vertices, it can be concluded that these graphs are not necessarily circulants. Furthermore, according to the theorem presented below, we prove that for every even value of $n$, there exists a connected graph that is generated by the aforementioned construction and is not circulant.
Assume that $n$ has the following prime factorization $n=2^{\alpha_1}p_2^{\alpha_2}\cdots p_k^{\alpha_k}$.
Let us introduce a notation for subsets of the divisor set $D\subseteq D_n$. We will denote $D_0$ as the set of divisors in $D$ where $n/d$ is an odd number, and $D_1$ as the set of divisors in $D$ where $n/d\in4\N+2$.
Also, for a positive integer $k$ and a set $A$ of positive integers, by $kA$ we will mean the set $\{ka\ |\ a\in A\}$.
 We use the following theorem from \cite{Ba22} for proving the mention statement.

\begin{theorem}
\cite{Ba22}  \label{cor:even=0} Let $\ICG_n(D)$ be an integral circulant graph.
The following statements are equivalent:
\begin{itemize}
\item[i)] Every $\lambda_j$ is even for odd $0\leq j\leq n-1$. 
\item[ii)] $D_0= 2D_1$.
\item[iii)] Every $\lambda_j=0$ for odd $0\leq j\leq n-1$.
\end{itemize}
\end{theorem}

\begin{theorem}
    For an arbitrary even $n$ and $d=\frac{n}{2^{\alpha_1}}$ the graph $(K_d\otimes K^*_{\frac n d})\otimes (K_d\otimes
K^*_{\frac n d})$ is not a circulant graph.
\end{theorem}
\begin{proof}
Suppose that $(K_d\otimes K^*_{\frac n d})\otimes (K_d\otimes
K^*_{\frac n d})$ is cirulant. Therefore, its spectrum can be denoted by $(\lambda_0,\ldots,\lambda_{n^2-1})$, which is given by the equation (\ref{ldef}). In other words, $(K_d\otimes K^*_{\frac n d})\otimes (K_d\otimes
K^*_{\frac n d})\simeq \ICG_{n^2}(D)$, for some $D\subseteq D_{n^2}$. Since $n$ is even and $d$ is odd, from (\ref{spectra tensor}), we see that every $\lambda_i$ is even. According to Theorem \ref{cor:even=0}, we have that $\lambda_j=0$ for odd $0\leq j\leq n^2-1$ and $D_0=2D_1$. Now, observe integral circulant graph $\ICG_{\frac{n^2}{2}}(D\setminus D_0)$ and denote its eigenvalues by $(\mu_0,\ldots,\mu_{\frac{n^2}{2}-1})$. For every even $0\leq j\leq n^2-1$, we see that 
$$
\lambda_j=\sum_{d\in D_1} (c(j,\frac{n^2}{d})+c(j,\frac{n^2}{2d}))+\sum_{d\in D\setminus (D_0\cup D_1)} c(j,\frac{n^2}{d}).
$$

For $d\in D_1$ we conclude that $\varphi(\frac{n^2}{d})=\varphi(\frac{n^2}{2d})$. Furthermore, we obtain  $\gcd(\frac{n^2}{d},j)=2\gcd(\frac{n^2}{2d},j)$ and also $t_{\frac{n^2}{d},j}=t_{\frac{n^2}{2d},j}$. This directly yields that
$c(j,\frac{n^2}{d})=c(j,\frac{n^2}{2d})$.
Similarly to the preceding case, we obtain that $\gcd(\frac{n^2}{d}, j) = 2 \gcd(\frac{n^2}{2d}, \frac{j}{2})$ and
$t_{\frac{n^2}{d},j} = t_{\frac{n^2}{2d},\frac{j}{2}}$, for $d\in D\setminus (D_0\cup D_1)$.
 Moreover, it holds that $\varphi(\frac{n^2}{d})=2^{2\alpha_1-1}\varphi(\frac{n^2}{2^{2\alpha_1}d})=2\cdot 2^{2\alpha_1-2}\varphi(\frac{n^2}{2^{2\alpha_1}d})=2\varphi(\frac{n^2}{2d})$. Therefore, we conclude that $c(j,\frac{n^2}{d})=2c(\frac{j}{2},\frac{n^2}{2d})$. According, to this discussion we get that 
 $$
\lambda_j=\sum_{d\in D_1} 2c(j,\frac{n^2}{d})+2\sum_{d\in D\setminus (D_0\cup D_1)} c(\frac{j}{2},\frac{n^2}{2d}).
 $$
Using a similar argument, we have that $c(j,\frac{n^2}{d})=c(\frac{j}{2},\frac{n^2}{2d})$, for $d\in D_1$, and thus 

 $$
\lambda_j=2\sum_{d\in D_1} c(\frac{j}{2},\frac{n^2}{2d})+2\sum_{d\in D\setminus (D_0\cup D_1)} c(\frac{j}{2},\frac{n^2}{2d})=2\mu_{\frac{j}{2}},
 $$
 for $0\leq \frac{j}{2}\leq n^2-1$.
This way we obtain that $\ICG_{\frac{n^2}{2}}(D\setminus D_0)$ has spectrum 
\begin{eqnarray*}
 Sp(\ICG_{\frac{n^2}{2}}(D\setminus D_0))=
\{(d-1)^2\frac{n^2}{2d^2}^{(1)},0^{(n^2-d^2-\frac{n^2}{2})},-(d-1)\frac{n^2}{2d^2}^{(2(d-1))},\frac{n^2}{2d^2}^{((d-1)^2)}\}.
\end{eqnarray*}
Through the repetition of this procedure $2\alpha_1$ times, which involves obtaining graphs with smaller orders $s$, we ultimately achieve a circulant graph with the corresponding spectrum.
\begin{eqnarray*}
\{(d-1)^2\frac{n^2}{2^{2\alpha_1}d^2}^{(1)},0^{(n^2-d^2-\frac{n^2}{2}-\cdots-\frac{n^2}{2^{2\alpha_1}})},-(d-1)\frac{n^2}{2^{2\alpha_1}d^2}^{(2(d-1))},\frac{n^2}{2^{2\alpha_1}d^2}^{((d-1)^2)}\}.
\end{eqnarray*}
Since $\frac{n^2}{2}+\cdots+\frac{n^2}{2^{2\alpha_1}}=\frac{n^2}{2^{2\alpha_1}}(1+2+\cdots+2^{2\alpha_1-1})=\frac{n^2(2^{2\alpha_1}-1)}{2^{2\alpha_1}}$,
we have that $n^2-d^2-\frac{n^2}{2}-\cdots-\frac{n^2}{2^{2\alpha_1}}=n^2(1-\frac{1}{2^{2\alpha_1}}-\frac{2^{2\alpha_1}-1}{2^{2\alpha_1}})=0$. This way, we obtain a regular connected circulant graph with three distinct eigenvalues in its spectrum, which means that it is strongly regular. However, this conclusion contradicts Theorem 15 stated in \cite{Ba22}, which asserts that a strongly regular integral circulant graph must have $0$ in its spectrum.
\qed
\end{proof}

Showing whether or not the graphs resulting from applying the tensor product operations to a certain number of graphs with specific properties are circulant or not is a challenging task. In \cite{GeSa99}, authors found specific classes of graphs that are either circulant or non-circulant graphs, obtained from a singe tensor product operation on two graphs of particular types.

It is worth mentioning that the graph $(K_2\otimes K^*_{\frac n 2})\otimes (K_2\otimes
K^*_{\frac n 2})$ belongs to the class of circulant graphs. However, it is important to note that this graph is disconnected strongly regular graph. 
Indeed, the spectrum of this graph is $\{\frac{n^2}{4}^{(2)},0^{(n^2-4)},-\frac{n^2}{4}^{(2)}\}$ and it contains only three distinct eigenvalues. Since it is a regular graph and the largest eigenvalue has a multiplicity of 2, we can conclude that it consists of two connected components whose spectra are $\{\frac{n^2}{4}^{(1)},0^{(\frac{n^2}{2}-2)},-\frac{n^2}{4}^{(1)}\}$. It is well known that the aforementioned spectrum is the spectrum of the complete bipartite graph $K_{\frac{n^2}{4},\frac{n^2}{4}}$, which is a circulant graph.

\bigskip

On the other hand, when $n\neq d^2$, we can consider the class of connected graphs $K_d \otimes K_{\frac{n}{d}}$ with a specific spectrum. This spectrum includes four distinct eigenvalues: $(d-1)\left(\frac{n}{d}-1\right)$ with multiplicity 1, $-(d-1)$, with multiplicity $\frac{n}{d}-1$, $-\left(\frac{n}{d}-1\right)$, with multiplicity ${d-1}$, and $1$, with multiplicity ${(d-1)\left(\frac{n}{d}-1\right)}$. Therefore, this class of graphs serves as an example where four distinct eigenvalues appear in their spectra.
According to Theorems 1 and 2 from \cite{GeSa99} this class of graphs is circulant if and only if 
$\gcd(d,\frac{n}{d})=1$.
In the theorem that follows, we establish that circulant graphs resulting from the described construction do not qualify as unitary Cayley graphs.
Unitary Cayley graphs with four eigenvalues will be analyzed in subsection 4.2. Namely, the following assertion holds.

\begin{theorem}
\label{thm:not unitary}
If there exists a divisor $d$ such that $\gcd(d,\frac{n}{d})=1$ for an arbitrary $n$, then $K_d\otimes K_{\frac n d}\simeq \ICG_{n}(\{d_1d_2 |\ d_1\in D_d,\
d_2\in D_{\frac{n}{d}}\})$.\end{theorem}
\begin{proof}
Since $K_d\simeq \ICG_d(\{d'\mid d\ |\
1\leq d'\leq d-1\})$, it is sufficient to prove that
$\ICG_n(D_1)\otimes \ICG_m(D_2)\simeq \ICG_{nm}(\{d_1d_2 |\ d_1\in D_1,\
d_2\in D_2\})$ for $\gcd(n,m)=1$ and
conclude that $K_d\otimes K_{\frac n d}$ is not unitary Cayley graph.
 Indeed, for $0\leq x\leq nm-1$
there exist unique $0\leq a\leq n-1$ and $0\leq b\leq m-1$ such that
$x\equiv am+bn \pmod {mn}$  and we can establish a bijection from $\ICG_n(D_1)\otimes
\ICG_m(D_2)$ onto $\ICG_{nm}(\{d_1d_2\ |\ d_1\in D_1,\ d_2\in D_2\})$ such
that $f(a,b)=x$. Moreover, for adjacent vertices $(a_1,b_1)$ and
$(a_2,b_2)$ from $\ICG_n(D_1)\otimes \ICG_m(D_2)$, we have that
$\gcd(a_1-a_2,n)=d_1\in D_1$ and $\gcd(b_1-b_2,m)=d_2\in D_2$.
We prove that $f(a_1,b_1)$ and $f(a_2,b_2)$ are adjacent in $\ICG_{nm}(\{d_1d_2\ |\ d_1\in D_1,\ d_2\in D_2\})$, i.e.,
we show that $\gcd((a_1-a_2)m+(b_1-b_2)n,nm)=d_1d_2$.
Indeed, let $d=\gcd((a_1-a_2)m+(b_1-b_2)n,nm)$ and without loss of generalization we can assume that
$d\mid m$, as $\gcd(m,n)=1$. Therefore, it holds that $d\nmid n$ and $d\mid b_1-b_2$. As $\gcd(b_1-b_2,m)=d_2$, we see that $d\mid d_2$. On the other hand,
as  $\gcd(b_1-b_2,m)=d_2$ we get that $d_2\mid (a_1-a_2)m+(b_1-b_2)n$ and similarly as $\gcd(a_1-a_2,n)=d_1$ we get that $d_1\mid (a_1-a_2)m+(b_1-b_2)n$. Finally, since
$\gcd(d_1,d_2)=1$, it holds that $d_1d_2\mid d$ and therefore $d_1d_2=d$.
The mapping $f$ is obviously bijection, thus making it an isomorphism between the graphs.
\qed
\end{proof}

Authors in \cite{ChRa18} (Section 3) present an example of integral circulant graphs $\ICG_n(D)$, where $n=p_1p_2p_3$ is a product of three primes and $D=\{1,p_i,p_j\}$ for $1\leq i\neq j\leq 3$, demonstrating four distinct eigenvalues in their spectra. Clearly, the graph $\ICG_{p_1p_2p_3}(1,p_i,p_j)$ is simply a special case of the graph $K_d\otimes K_{\frac{n}{d}}$, where $n=p_1p_2p_3$, $d=p_k$, $1\leq i\neq k\neq j\leq 3$, and $\gcd(d,\frac{n}{d})=1$.


\bigskip

A well known fact is that when we apply the complement operation to a class of regular graphs results in a class of regular graphs whose spectra contain an equal or fewer number of distinct eigenvalues compared to the original graphs.
However, the resulting class of graphs does not need to be connected.

Let $\lambda_i$, for $0\leq i\leq n-1$, be the eigenvalues of $K_d\otimes K_{\frac{n}{d}}$, for $\gcd(d,n/d)=1$, where $\lambda_0$ is the regularity of the graph. The eigenvalues of the complement graph of the graph
$\ICG_n(D)$ are $-1-\lambda_i$, $1\leq i\leq n-1$ and $n-1-\lambda_0$ (for example, see Lemma 8.5.1 from \cite{GoRo01}), whence  we conclude that the complement graph of $K_d\otimes K_{\frac{n}{d}}$ has the following spectrum
$\{d+\frac{n}{d}-2^{(1)},d-2^{(\frac{n}{d}-1)},\frac{n}{d}-2^{(d-1)},-2^{((d-1)(\frac{n}{d}-1))}\}$. Given that the eigenvalue $d+\frac{n}{d}-2$ stands as the largest single eigenvalue in the spectrum, the subsequent chain of inequalities holds: $d+\frac{n}{d}-2 > d-2 > \frac{n}{d}-2 > -2$, considering that $d > \frac{n}{d}$. Consequently, we can deduce that $\overline{K_d\otimes K_{\frac{n}{d}}}$ constitutes a connected graph featuring four distinct eigenvalues in its spectrum.
By showing that $K_d \otimes K_{\frac{n}{d}}$ is not a self-complementary graph, we establish the existence of a novel class of circulant graphs possessing four distinct eigenvalues in their spectra. If we suppose that it is self-complementary graph, then either the eigenvalue $d-2$ or the eigenvalue $\frac{n}{d}-2$ of $\overline{K_d\otimes K_{\frac{n}{d}}}$ is equal to the eigenvalue $1$ of $K_d\otimes K_{\frac{n}{d}}$. In both cases, we observe that the regularity of 
$\overline{K_d\otimes K_{\frac{n}{d}}}$ is equal to $\frac{n}{3}-1$, which is less than $\frac{n-1}{2}$, and this contradiction arises as a result. 
Furthermore, based on the proof of Theorem \ref{thm:not unitary}, we can deduce that $\overline{K_d \otimes K_{\frac{n}{d}}}$ is isomorphic to $\ICG_n(D_n\setminus\{d_1d_2 |\ d_1\in D_{d},\
d_2\in D_{\frac{n}{d}})\}$. This finding leads us to the conclusion that these graphs do not meet the criteria to be classified as unitary Cayley graphs.

\medskip

In the following theorem, we present a class of circulant graphs whose spectrum contains four distinct eigenvalues of composite ordersimilar to the class discussed in the preceding text. However, this class of graphs contains irrational eigenvalues and can be derived from the class of strongly regular graphs found in Theorem \ref{thm:strongly-regular prime} through certain graph operations.
Also, for a positive integer $k$ and a set $A$ of positive integers, by $k+A$ we will mean the set $\{k+a\ |\ a\in A\}$.

\begin{theorem}
    Let $G(p;S)$ be a strongly regular circulant graph with a prime order $p$ and a set of symbols $S\subseteq\{1,\ldots, p-1\}$. Then, if $n$ is a composite number divisible by $p$, the spectrum of the graph $G(n;\cup_{i=0}^{\frac{n}{p}-1}(ip+S))$ contains four distinct eigenvalues. 
\end{theorem}
\begin{proof}
Since $G(p;S)$ is a strongly regular circulant graph with a prime order $p$, as stated in Theorem \ref{thm:strongly-regular prime}, it follows that $S$ corresponds to the set of quadratic residues modulo $p$.
Let $\lambda_0, \ldots, \lambda_{n-1}$ denote the eigenvalues of $G(n;\cup_{i=0}^{\frac{n}{p}-1}(ip+S))$, as given by  equation (\ref{eq:eigenvalues unwigted}). We can deduce that $|ip+S|=|S|=\frac{p-1}{2}$ for every $0\leq i\leq \frac{n}{p}-1$, thereby establishing $\lambda_0=\frac{n}{p}\cdot\frac{p-1}{2}$. Suppose now that $1\leq i\leq n-1$. According to 
(\ref{eq:eigenvalues unwigted}), it holds that
$$
\lambda_i=\sum_{j=0}^{\frac{n}{p}-1}\sum_{s\in S} \omega_n^{(s+jp)i}=\sum_{j=0}^{\frac{n}{p}-1}\sum_{s\in S} \omega_n^{si}\omega_n^{jpi}=(\sum_{j=0}^{\frac{n}{p}-1}\omega_n^{jpi})(\sum_{s\in S} \omega_n^{si})=(\sum_{j=0}^{\frac{n}{p}-1}\omega_{\frac{n}{p}}^{ji})(\sum_{s\in S} \omega_n^{si}).
$$

If $\omega_{\frac{n}{p}}^{i}\neq 1$, which is the case when $\frac{n}{p}\nmid i$, then we can observe that $(\sum_{j=0}^{\frac{n}{p}-1}\omega_{\frac{n}{p}}^{ji})(\sum_{s\in S} \omega_n^{si})=\frac{\omega_{\frac{n}{p}}^{i\frac{n}{p}}-1}{\omega_{\frac{n}{p}}^{i}- 1}\sum_{s\in S} \omega_n^{si}=0$. Consequently, we have demonstrated that $0$ is an eigenvalue of the graph $G(n;\cup_{i=0}^{\frac{n}{p}-1}(ip+S))$, and $\lambda_i=0$ whenever $\frac{n}{p}\nmid i$ for $1\leq i\leq n-1$. As a result, $|{1\leq i\leq n-1\ |\ \frac{n}{p}\nmid i}|=n-|{1\leq i\leq n\ |\ \frac{n}{p}\mid i}|=n-\frac{n}{n/p}=n-p$, which represents the multiplicity of the eigenvalue $0$.

If $\omega_{\frac{n}{p}}^{i}= 1$, which occurs when $\frac{n}{p}\mid i$, then we have $\sum_{j=0}^{\frac{n}{p}-1}\omega_{\frac{n}{p}}^{ji}=\frac{n}{p}$. 
Hence, we can establish that $\sum_{j=0}^{\frac{n}{p}-1}\omega_{\frac{n}{p}}^{ji}\sum_{s\in S} \omega_n^{si}=\frac{n}{p}\sum_{s\in S} \omega_p^{s{i_1}}$, where $i=\frac{n}{p}i_1$.
According to the discussion following Theorem \ref{thm:strongly-regular prime}, if $i_1$ is a quadratic residue modulo $p$, then $\sum_{s\in S} \omega_p^{s{i_1}}=\tau=\frac{\sqrt{p}-1}{2}$, while if $i_1$ is a non-quadratic residue modulo $p$, then $\sum_{s\in S} \omega_p^{s{i_1}}=\theta=\frac{-\sqrt{p}-1}{2}$. Therefore, if $\frac{ip}{n}$ is a quadratic residue then $\lambda_i=\frac{n}{p}\tau=\frac{n(\sqrt{p}-1)}{2p}$, while if $\frac{ip}{n}$ is a non-quadratic residue then $\lambda_i=\frac{n}{p}\theta=-\frac{n(\sqrt{p}+1)}{2p}$.
Since we have previously established that $m_{\tau}=m_{\theta}=\frac{p-1}{2}$, we can conclude that $\lambda_i$  also has a multiplicity of $\frac{p-1}{2}$, where $i_1=\frac{ip}{n}$ and $i_1$ is a (non-)residue modulo $p$.

    \qed
\end{proof}

Let us enumerate the vertices of the graph $G(n;\cup_{i=0}^{\frac{n}{p}-1}(ip+S))$ as $0, 1, \ldots, n-1$, and partition them into the classes $C_0, C_1, \ldots, C_{\frac{n}{p}-1}$ such that ${ip, ip+1, \ldots, ip+(p-1)} \in C_i$.
Let $u$ and $v$ be two vertices such that $u \in C_i$ and $v \in C_j$, for some $0 \leq i, j \leq \frac{n}{p}-1$. This implies that $u = ip + r_1$ and $v = jp + r_2$, for some $0 \leq r_1, r_2 \leq p-1$, and without loss of generality, we assume that $i \geq j$. Consequently, it holds that $u-v \in (i-j)p + S$ if and only if $r_1 - r_2 \in S$. In particular, if we set $u, v \in C_i$ and establish the mapping $f(u) = r_1$ and $f(v) = r_2$ ($f$ maps the elements of the class $C_i$ to their residues), it can be seen that a subgraph induced by the class $C_i$ is isomorphic to $G(p; S)$ with respect to the mapping $f$. Moreover, we can conclude that $G(n;\cup_{i=0}^{\frac{n}{p}-1}(ip+S))$ consists of $\frac{n}{p}$ copies of $G(p; S)$, and two vertices from distinct copies congruent modulo $p$ have the same neighborhood. 
Furthermore, we can represent each vertex $u$ in $C_i$ using a tuple where the first position corresponds to the class it belongs to (in this case, $i$), and the second position represents its position within the class $C_i$ determined by its residue $r$ modulo $p$. Hence, we can state that two vertices $(i,r_1)$ and $(j,r_2)$ are adjacent if and only if $r_1 - r_2 \in S$. In other words, in the graph $G(n;\cup_{i=0}^{\frac{n}{p}-1}(ip+S))$, the vertices $(i,r_1)$ and $(j,r_2)$ are adjacent if and only if $r_1$ and $r_2$ are adjacent in $G(p; S)$ for all $0 \leq i, j \leq \frac{n}{p}-1$.
Considering the set $\{i\ |\ 0\leq i\leq \frac{n}{p}-1\}$ as the vertices of the graph $K^{*}_{\frac{n}{p}}$, where every two vertices are adjacent, we can conclude that $(i,r_1)$ and $(j,r_2)$ are adjacent in $G(n;\cup_{i=0}^{\frac{n}{p}-1}(ip+S))$ if and only if $i$ and $j$ are adjacent in $K^{*}_{\frac{n}{p}}$, and $r_1$ and $r_2$ are adjacent in $G(p; S)$. This can be expressed as the isomorphism $G(n;\cup_{i=0}^{\frac{n}{p}-1}(ip+S))\simeq K^{*}_{\frac{n}{p}}\otimes G(p; S)$.

\medskip

In this context, we aim to discover a new classes of integral circulant graphs with composite order that possess four distinct eigenvalues in their spectra.
\begin{theorem}
\label{thm: second class}
   Integral circulant graph $\ICG_n(D)$ posses four distinct eigenvalues in its spectrum, whenever $n$ is a composite integer
and $D = \{d \in D_n\ |\ k \nmid d\}\cup km D_{\frac{n}{km}}$, 
for some divisors $n - 1 \geq k \geq 2$ of $n$ and $\frac{n}{k} - 1 \geq m \geq 2$ of $\frac{n}{k}$ .
\end{theorem}
\begin{proof}
First, we compute the eigenvalues of $\mu_0,\ldots,\mu_{rs-1}$, which is given by the equation (\ref{ldef}), of the graph $\ICG_{rs}(rD_s)$, for arbitrary positive integers $r\geq 2$ and $s\geq 1$. Given that $\gcd({d\ |\ d\in rD_s})=r\geq 2$, according to Lemma \ref{lem:ICG unconnected}, it follows that $\ICG_{rs}(rD_s)$ is disconnected and consists of $r$ connected components, each of which is isomorphic to the complete graph $\ICG_{s}(D_s)$. Furthermore, considering that the spectrum of $\ICG_{s}(D_s)$  is represented as 
$s-1,\underbrace{-1,\ldots,-1}_{s-1}$
and the spectrum of $\ICG_{rs}(rD_s)$ is composed of $r$ copies of the spectrum of $\ICG_{s}(D_s)$, we deduce that $\mu_j=s-1$, for $s\mid j$, and $\mu_j=-1$, for $s\nmid j$. 

Let $\lambda_0,\ldots, \lambda_{n-1}$ be the spectrum of $\ICG_n(D)$, $\nu_0,\ldots, \nu_{n-1}$ be the spectrum of $\ICG_n(\{d \in D_n\ |\ k \nmid d\})$ and $\eta_0,\ldots, \eta_{n-1}$ be the spectrum of $\ICG_n(km D_{\frac{n}{km}})$, obtained by (\ref{ldef}). It is clear that $\lambda_i=\nu_i+\eta_i$, for $0\leq i\leq n-1$. Since $\ICG_n(\{d \in D_n\ |\ k \nmid d\})$ represents a regular graph, we see that its complement $\ICG_n(\{d \in D_n\ |\ k \mid d\})\simeq \ICG_n(kD_{\frac{n}{k}})$ has spectrum $n-1-\nu_0,-1-\nu_1,\ldots, -1-\nu_{n-1}$.
According to the preceding discussion, we get that $n-1-\nu_0=\frac{n}{k}-1$, $-1-\nu_i=\frac{n}{k}-1$, for $\frac{n}{k}\mid i$, and  $-1-\nu_i=-1$, for $\frac{n}{k}\nmid i$, $1\leq i\leq n-1$. In conclusion, this implies that $\nu_0=n-\frac{n}{k}$, $\nu_i=-\frac{n}{k}$, for $\frac{n}{k}\mid i$, and  $\nu_i=0$, for $\frac{n}{k}\nmid i$, $1\leq i\leq n-1$. 
Similarly, we can compute the spectrum of  $\ICG_n(km D_{\frac{n}{km}})$ as follows:  
$\eta_i=\frac{n}{km}-1$, for $\frac{n}{km}\mid i$, and $\eta_i=-1$, for $\frac{n}{km}\nmid i$, $0\leq i\leq n-1$.
Finally, as $\lambda_i=\nu_i+\eta_i$, for $0\leq i\leq n-1$, it can be seen that $\lambda_0=n-\frac{n}{k}+\frac{n}{km}-1$, $\lambda_i=-\frac{n}{k}+\frac{n}{km}-1$, for $\frac{n}{k}\mid i$,
$\lambda_i=\frac{n}{km}-1$, for $\frac{n}{km}\mid i,\ \frac{n}{k}\nmid i $, and $\lambda_i=-1$, for $\frac{n}{km}\nmid i$, for $1\leq i\leq n-1$. The multiplicities of the particular eigenvalues are given by the following formula
$$
\{n-\frac{n}{k}+\frac{n}{km}-1^{(1)},\frac{n}{km}-1^{(km-k)},-\frac{n}{k}+\frac{n}{km}-1^{(k-1)},-1^{(n-mk)}\}.
$$

 \qed   
\end{proof}

It is worth noting that the class of graphs $\ICG_{p^{\alpha}}(1,p^{\alpha-1})$, as described in \cite{ChRa18} by Theorem 4.1, represents only a specific case within the broader class of graphs defined in Theorem \ref{thm: second class}. In the case of $n=p^\alpha$, $k=p$, and $m=p^{\alpha-2}$, we have $\{d \in D_n\ |\ k \nmid d\}=\{1\}$ and $km D_{\frac{n}{km}}=p^{\alpha-1}D_p=\{p^{\alpha-1}\}$. Consequently,  $\ICG_n(\{d \in D_n\ |\ k \nmid d\}\cup km D_{\frac{n}{km}})=\ICG_{p^{\alpha}}(1,p^{\alpha-1})$. Similarly, for $n=p^{\alpha}$, $m=p$ and $k=p^s$  (where $1<s\leq \alpha-2$), we can deduce that $\{d \in D_n\ |\ k \nmid d\}=\{1,p,\ldots,p^{s-1}\}$. We also have 
$km D_{\frac{n}{km}}=p^{s+1}D_{p^{\alpha-s-1}}=\{p^{s+1},p^{s+2},\ldots,p^{\alpha-1}\}$ and hence $\ICG_{p^{\alpha}}(D_{p^{\alpha}}\setminus \{p^s\})$ is an example of a graph which exhibits four distinct eigenvalues in its spectrum, as derived in  \cite{ChRa18} through Theorem 4.3. Moreover, the statement of Theorem 4.5 from \cite{ChRa18} directly follows from Theorem \ref{thm: second class} for $n=p^\alpha$, $k=p^{\alpha-2}$ and $m=p$.
Furthermore, the statement of Theorem 5.1 from \cite{ChRa18} directly follows from Theorem \ref{thm: second class} for even composite $n$, $k=d$ and $m=\frac{n}{2d}$, where $d$ is a divisor of $\frac{n}{2}$. Thus, we have demonstrated that all the classes of circulant graphs discussed in \cite{ChRa18} are merely specific instances of the class of graphs obtained in the theorem mentioned above.

It is important to highlight that the graph mentioned in Theorem \ref{thm: second class} can be expressed as the union of graphs: the strongly regular graph $\ICG_n({d \in D_n\ |\ k \nmid d})$ and $km$ copies of the complete graph $K_{km}$. Furthermore, referring to equation (\ref{kronecker_complete}) and the subsequent discussion, we observe that the graph $\ICG_n({d \in D_n\ |\ k \nmid d})$ can be represented as the tensor product of specific graphs.

\bigskip

In the following theorem we present another class of integral cicrculant graphs with composite order whose spectrum contains four distinct eigenvalues. It is noteworthy that this particular class of graphs can be viewed as a union of two distinct graphs, both of which are subgraphs of strongly regular integral circulant graphs of specific orders.
\begin{theorem}
    Let $n$ be an composite even number and $k$ be an odd divisor of $n$, such that $n-1\geq k\geq 2$. Then, the spectrum of the integral circulant graph $\ICG_n(D)$, for $D = \{d \in D_n\ |\ d\in 2\N,\ k \nmid d\}\cup \{d \in D_n\ |\ d\in 2\N+1,\ k \mid d\}$, posses four distinct eigenvalues. 
\end{theorem}
\begin{proof}
Let $\lambda_0,\ldots, \lambda_{n-1}$ be the spectrum of $\ICG_n(D)$, $\nu_0,\ldots, \nu_{n-1}$ be the spectrum of $\ICG_n(\{d \in D_n\ |\ d\in 2\N,\ k \nmid d\})$ and $\eta_0,\ldots, \eta_{n-1}$ be the spectrum of $\ICG_n(\{d \in D_n\ |\ d\in 2\N+1,\ k \mid d\})$, obtained by (\ref{ldef}). It is clear that $\lambda_i=\nu_i+\eta_i$, for $0\leq i\leq n-1$.
Based on Lemma \ref{lem:ICG unconnected}, it can be observed that the graph $\ICG_n(\{d \in D_n\ |\ d\in 2\mathbb{N},\ k \nmid d\})$ is isomorphic to two copies of the graph $\ICG_{n/2}(\{d \in D_{\frac{n}{2}}\ |\ k \nmid d\})$. Additionally, as deduced from the proof of Theorem \ref{thm: second class}, the spectrum of this graph consists of three distinct eigenvalues: $\{\frac{n}{2}-\frac{n}{2k}, -\frac{n}{2k}, 0\}$. Consequently, we have $\nu_0=\nu_{\frac{n}{2}}=\frac{n}{2}-\frac{n}{2k}$, $\nu_i=-\frac{n}{2k}$ for $\frac{n}{2k}\mid i$ and $\frac{n}{2}\nmid i$, and $\nu_i=0$ otherwise.    

In the same fashion, we can conclude that $\ICG_n(\{d \in D_n\ |\ d\in 2\N+1,\ k \mid d\})$ is isomorphic to $k$ copies of
$\ICG_{\frac{n}{k}}(\{d \in D_{\frac{n}{k}}\ |\ 2\nmid d\})$. The spectrum of 
$\ICG_{\frac{n}{k}}(\{d \in D_{\frac{n}{k}}\ |\ 2\nmid d\})$ contains three distinct eigenvalues $\{\frac{n}{k}-\frac{n}{2k}, -\frac{n}{2k}, 0\}$. Moreover, it holds that $\eta_i=\frac{n}{2k}$, for $\frac{n}{k}\mid i$, $\eta_i=-\frac{n}{2k}$ for
$\frac{n}{2k}\mid i$ and $\frac{n}{k}\nmid i$, and $\eta_i=0$, otherwise. Therefore, we see that $\lambda_0=\frac{n}{2}-\frac{n}{2k}+\frac{n}{2k}=\frac{n}{2}$, $\lambda_{\frac{n}{2}}=\frac{n}{2}-\frac{n}{2k}-\frac{n}{2k}=\frac{n}{2}-\frac{n}{k}$,
$\lambda_i=-\frac{n}{k}$, for $\frac{n}{2k}\mid i$, $\frac{n}{2}\nmid i$ and $\frac{n}{k}\nmid i$, and $\lambda_i=0$ otherwise.

    \qed
\end{proof}

Observing the complement of the graph $\ICG_n(D)$, for $D = \{d \in D_n\ |\ d\in 2\N,\ k \nmid d\}\cup \{d \in D_n\ |\ d\in 2\N+1,\ k \mid d\}$, we can found that its spectrum consists of four distinct eigenvalues:
$\mu_0=n-1-\frac{n}{2}=\frac{n}{2}-1$, $\mu_{\frac{n}{2}}=-1-\frac{n}{2}+\frac{n}{k}$,
$\mu_i=\frac{n}{k}-1$, for $\frac{n}{2k}\mid i$, $\frac{n}{2}\nmid i$ and $\frac{n}{k}\nmid i$, and $\mu_i=-1$ otherwise.

\medskip

In this section, thus far, we have identified classes of graphs, both circulant and non-circulant, whose spectra exhibit four distinct eigenvalues having composite order. As a result, we are now prepared to provide a characterization of all circulant graphs with prime order that possess four distinct eigenvalues in the subsequent statement.

\begin{theorem}
\label{thm:four eigenvalues prime}
For a prime number $p$ the spectrum of a circulant graph $G(p;S)$ posses four distinct eigenvalues in its spectrum if and only if $S$ is a set of all cubic residues modulo $p$
 or all cubic non-residues modulo $p$ and $p\in 3\N+1$.
    
\end{theorem}
\begin{proof}
Let $G(p;S)$ exhibits four distinct eigenvalues in its spectrum. Since $\lambda_0=|S|$, it follows that the sequence of eigenvalues $\lambda_1,\lambda_2,...,\lambda_{p-1}$, given by (\ref{eq:eigenvalues unwigted}), must consist of exactly three distinct eigenvalues. We retain the same notation as in the proof of Theorem \ref{thm:strongly-regular prime}.
Since $S_1=S$, the number of distinct eigenvalues in the spectrum of $G(p;S)$ is equal to four if and
only if $\{S_i|\ 1\leq i\leq p-1\}=\{S,T,R\}$ for some $T,R\subseteq \{1,\ldots,p-1\}$.
Similarly as in the proof of Theorem \ref{thm:strongly-regular prime}, we conclude that $|S_i|=|S|=|T|=|R|$, for $1\leq i\leq p-1$.
Furthermore, for a given $s\in S$, we can conclude that $\{r_{i,s}|\ 1\leq i\leq p-1\}=\{1,\ldots,p-1\}$, which implies that $S\cup T\cup R=\{1,\ldots,p-1\}$.

We will now show  that the sets $S$, $T$ and $R$ are pairwise disjoint.
Suppose there exists an element $c$ that belongs to the intersection of $S$, $T$, and $R$. Following a similar approach as in the proof of Theorem \ref{thm:strongly-regular prime}, we can establish that $S=T=R=\{1,\ldots,p-1\}$. However, this leads to a contradiction. 
Now, without loss of generality, we can assume that there exists some $c\in S\cap T$ and $c\not\in R$.
This means that $c\in S_i$ for $\{1\leq i\leq p-1\ |\ S_i\in\{S,T\}\}$. 
Therefore, for each such $i$, there exists $s\in S$ such that $c=r_{i,s}$, and consequently, $s\equiv_p c\cdot i^{-1}$. 
This implies that $\{c\cdot i^{-1}\ |\ 1\leq i\leq p-1,\ S_i\in\{S,T\}\}\subseteq S$.
Let $1\leq j\leq p-1$ be an arbitrary index such that $S_j=R$. If $c\cdot j^{-1}\in S$, then $j\cdot c\cdot j^{-1}\in S_j$, and hence $c\in R$, leading to a contradiction. Thus, we conclude that $\{c \cdot i^{-1} | 1 \leq i \leq p-1, S_i \in \{S,T\}\} = S$. Let $i$ and $j$ be  indices such that $S_i=S$ and $S_j=T$. Since, $S_i\neq S_j$, then there exist elements $ca^{-1}$ and $cb^{-1}$ from the set $\{c\cdot i^{-1}\ |\ 1\leq i\leq p-1,\ S_i\in\{S,T\}\}$ such that $cia^{-1}\not\equiv_{p} cjb^{-1}$.
We can distinguish three cases: $S_a=S$ and $S_b=S$, $S_a=S$ and $S_b=T$, and $S_a=T$ and $S_b=T$. Suppose that $S_a=S$ and $S_b=S$. From  $cia^{-1}\not\equiv_{p} cjb^{-1}$, we have that $cai^{-1}\not\equiv_{p} cbj^{-1}$, which means that $S_a\neq S_b$, as  $cai^{-1}\in S_a$ and $cbj^{-1}\in S_b$. This leads us to a contradiction. In the similar fashion the same conclusion can be deduced for the case $S_a=T$ and $S_b=T$. Finally, suppose that $S_a=S$ and $S_b=T$.
From  $cia^{-1}\not\equiv_{p} cjb^{-1}$, we have that $cba^{-1}\not\equiv_{p} cji^{-1}$. This implies that $S_b\neq S_j$, as $cba^{-1}$ belongs to $S_b$ and $cji^{-1}$ belongs to $S_j$. However, this contradicts our assumption that $S_b=S_j=T$.

Since $|S|=|T|=|R|$, $S\cup T\cup R=\{1,\ldots,p-1\}$ and $S$, $T$ and $R$ are pairwise disjoint, we obtain that $|S|=|T|=|R|=\frac{p-1}{3}$. Based on Lemma \ref{eq:eigenvalues equality}, the equation $\lambda_i=\lambda_j$ implies $\prod_{s\in S}r_{i,s}=\prod_{s\in S}r_{j,s}$, which further leads to $\prod_{s\in S}is\equiv_p\prod_{s\in S}js$ and subsequently $i^{\frac{p-1}{3}}\equiv_pj^{\frac{p-1}{3}}$.
By examining the values of $i^{\frac{p-1}{3}}$ modulo $p$, which can take the forms $1$, $x$, or $x^2$, for some $x$, $x^2+x+1\equiv_p 1$, we can deduce that if $\lambda_i=\lambda_j$, then the values of $i$ and $j$ are both  cubic residues modulo $p$ (as stated in Theorem \ref{thm:residues}), or they satisfy the condition $i^{\frac{p-1}{3}}\equiv_pj^{\frac{p-1}{3}}\equiv_p x$ or $i^{\frac{p-1}{3}}\equiv_pj^{\frac{p-1}{3}}\equiv_p x^2$. If  one of the last two conditions is satisfied, then it is easy to see that $i\equiv_p xi_1$ or $i\equiv_p x^2i_1$, respectively, for some cubic residue $i_1$ modulo $p$. 
Since $S_1=S$ it follows $\{i|\ S_i=S,\
1\leq i\leq p-1\}\subseteq \{i|\
\big{[}\frac{i}{p}\big{]}_3=1,\
1\leq i\leq p-1\}$. Similarly, we conclude that $\{i|\ S_i=T,\
1\leq i\leq p-1\}\subseteq \{x\cdot i|\
\big{[}\frac{i}{p}\big{]}_3=1,\
1\leq i\leq p-1\}$ and $\{i|\ S_i=R,\
1\leq i\leq p-1\}\subseteq \{x^2\cdot i|\
\big{[}\frac{i}{p}\big{]}_3=1,\
1\leq i\leq p-1\}$. 
Given that $\{i|\
\big{[}\frac{i}{p}\big{]}_3=1,\
1\leq i\leq p-1\}=|\{x\cdot i|\
\big{[}\frac{i}{p}\big{]}_3=1,\
1\leq i\leq p-1\}|=|\{x^2\cdot i|\
\big{[}\frac{i}{p}\big{]}_3=1,\
1\leq i\leq p-1\}|=\frac{p-1}{3}$ and $\{i|\ S_i=S,\
1\leq i\leq p-1\}\cup\{i|\ S_i=T,\
1\leq i\leq p-1\}\cup\{i|\ S_i=R,\
1\leq i\leq p-1\}=\{1,\ldots, p-1\}$, we have that $\{i|\ S_i=S,\
1\leq i\leq p-1\}= \{i|\
\big{[}\frac{i}{p}\big{]}_3=1,\
1\leq i\leq p-1\}$, $\{i|\ S_i=T,\
1\leq i\leq p-1\}= \{x\cdot i|\
\big{[}\frac{i}{p}\big{]}_3=1,\
1\leq i\leq p-1\}$ and $\{ i|\ S_i=R,\
1\leq i\leq p-1\}= \{x^2\cdot i|\
\big{[}\frac{i}{p}\big{]}_3=1,\
1\leq i\leq p-1\}$.

Suppose there exists $s\in S$ such that $\big{[}\frac{s}{p}\big{]}_3=1$. For every $1\leq i\leq p-1$ such that $\big{[}\frac{i}{p}\big{]}_3=1$, we have that 
$S_i=S$ and $r_{i,s}\in S_i$, thereby implying that $r_{i,s}\in S$. 
In other words, we have that $\{r_{i,s}\ |\ \big{[}\frac{i}{p}\big{]}_3=1,\  1\leq i\leq p-1\}\subseteq S$. 
Since 
$\big{[}\frac{is}{p}\big{]}_3=\big{[}\frac{i}{p}\big{]}_3\big{[}\frac{s}{p}\big{]}_3=1$, we conclude that $\{r_{i,s}\ |\ \big{[}\frac{i}{p}\big{]}_3=1,\ 1\leq i\leq p-1 \}\subseteq \{i|\
\big{[}\frac{i}{p}\big{]}_3=1,\
1\leq i\leq p-1\}$. 
Moreover, from the fact that $i\neq j$ implies $r_{i,s}\neq r_{j,s}$, for $1\leq i,j\leq p-1$,
we further get that $\{r_{i,s}\ |\ \big{[}\frac{i}{p}\big{]}_3=1,\  1\leq i\leq p-1\}= \{i|\
\big{[}\frac{i}{p}\big{]}_3=1,\
1\leq i\leq p-1\}$. 
Finally, from the preceding discussion it can be concluded that 
$\{i|\ \big{[}\frac{i}{p}\big{]}_3=1,\ 1\leq i\leq p-1\}\subseteq S$ and since $|\{i|\
\big{[}\frac{i}{p}\big{]}_3=1,\
1\leq i\leq p-1\}|=|S|=\frac{p-1}{3}$, it holds that $\{i|\
\big{[}\frac{i}{p}\big{]}_3=1,\
1\leq i\leq p-1\}= S$. 
If we assume that there exists $s\in S$ such that $s\equiv_{p}x\cdot s_1$ or  $s\equiv_{p}x^2\cdot s_1$, where
$\big{[}\frac{s_1}{p}\big{]}_3=1$, it can be proven in a similar fashion $S= \{x \cdot i|\
\big{[}\frac{i}{p}\big{]}_3=1,\
1\leq i\leq p-1\}$ or $S= \{x^2 \cdot i|\
\big{[}\frac{i}{p}\big{]}_3=1,\
1\leq i\leq p-1\}$. For all the obtained cases concerning the values of $S$, it can be concluded that the graphs $G(p;S)$ are mutually isomorphic.

    \qed
\end{proof}

We can proceed with determining the spectrum of the graph $G(p;S)$, where $p$ is a prime number of the form $3k+1$ and $S$ represents the set of cubic residues modulo $p$. In the previous theorem, we have established that the eigenvalues $\lambda_1,\ldots,\lambda_{p-1}$ of $G(p;S)$ include three distinct values, denoted as $\theta$, $\tau$, and $\eta$. These values can be calculated by considering the following sums: $\theta=\sum_{s\in S} \omega_{p}^s$, $\tau=\sum_{s\in S} \omega_{p}^{xs}$, and $\eta=\sum_{s\in S} \omega_{p}^{x^2s}$, where $x$ represents a cubic non-residue modulo $p$.
Given that $p\in 3\N+1$, according to  well-known theorem of Fermat, there exist unique integers $a$ and $b$, up to sign, such that $4p=a^2+27b^2$ (for more details see \cite{Cox89}).
The sums $\theta$, $\tau$, and $\eta$  are closely related to so called Gauss cubic sums and represent  the roots of the polynomial $t^3+t^2-\frac{p-1}{3}t-\frac{ap+3p-1}{27}$ (this result can be found in \cite[p.~460--461]{Hasse50}).

\subsection {Construction of regular graphs with four distinct eigenvalues using line operator}

In this section, we find novel classes of graphs whose spectra posses four distinct eigenvalues by applying the line graph operator $L$ on the class of unitary Cayley graphs.
First, we aim to characterize all unitary Cayley graphs whose spectra exhibit four distinct eigenvalues.

\begin{theorem}
\label{thm:ucg-4 eigenvalues} Unitary Cayley graph $X_n$ has four
distinct eigenvalues if and only if $n$ is the product of two
distinct primes.
\end{theorem}
\begin{proof}
Suppose that $X_n$ has exactly four distinct eigenvalues and that
there exists prime $p_i$ in the prime factorization of
$n=p_1^{\alpha_1}\cdots p_k^{\alpha_k}$, such that $\alpha_i\geq 2$.
If $k=1$, according to Theorem \ref{thm:sruc}, $X_n$ is strongly
regular which is a contradiction, so we assume that $k\geq 2$.

Let $S=\{p_1^{\beta_1}\cdots p_k^{\beta_k}<n\ |\
\beta_i\in\{\alpha_i-1,\alpha_i\}\}$ and $m=p_1p_2\cdots p_k$. Since
$k\geq 2$, we obtain $|S|\geq 3$. For $j\in S$, where
$j=p_1^{\beta_1}\cdots p_k^{\beta_k}$ we conclude that
$t_{n,j}=\frac{n}{\gcd(n,j)}= p_1^{\alpha_1-\beta_1}\cdots
p_k^{\alpha_k-\beta_k}$, $\varphi(t_{n,j})=\prod_{\beta_i=\alpha_i-1} (p_i-1)$ and
\begin{eqnarray*}
\lambda_j=c(j,n)&=&\mu(t_{n,j})\prod_{\beta_i=\alpha_i-1}p_i^{\alpha_i-1}\prod_{\beta_i=\alpha_i}p_i^{\alpha_i-1}(p_i-1)=\mu(t_{n,j})\frac{n}{m}\prod_{\beta_i=\alpha_i}(p_i-1)\neq
0.
\end{eqnarray*}
By the definition of $j\in S$, there exists a $k-$tuple
$(\beta_1,\ldots,\beta_k)$ such that $j=p_1^{\beta_1}\cdots
p_k^{\beta_k}$, so it can be written in the following form
$j=j(\beta_1,\ldots,\beta_k)$, for a given $n$. Also, by $j_i$, $1\leq i\leq k$, we
denote the index $j\in S$ such that
$j_i=j(\alpha_1-1,\alpha_2-1,\ldots,\alpha_i-1,\alpha_{i+1},\alpha_{i+2},\ldots,\alpha_k)$.
It is easy to see that for $1\leq i_1<i_2\leq k$, it holds that
$j_{i_1}>j_{i_2}$ and $|\lambda_{j_{i_1}}|>|\lambda_{j_{i_2}}|$.
Therefore, if $k\geq 3$ we have at least three distinct values among
the eigenvalues $\lambda_{j_1}, \lambda_{j_2},\ldots, \lambda_{j_k}$
and together with the regularity $\lambda_0$ and
$\lambda_1=\mu(n)=0$, we conclude that $X_n$ has at least five
eigenvalues, which is a contradiction. If $k=2$ then it is obvious
that $|S|=3$. These eigenvalues are equal to $\frac{n}{m},
-\frac{n}{m}(p_1-1), -\frac{n}{m}(p_2-1)$, so they are mutually
distinct. Similarly, together with the regularity $\lambda_0$ and
$\lambda_1=\mu(n)=0$, we conclude that $X_n$ has at least five
eigenvalues, which is a contradiction.

Now, suppose that $n$ is a square-free number. If $n$ is prime, then
$X_n$ has exactly two distinct eigenvalues, which is a
contradiction. So, we assume that $n$ has the following prime
factorization $n=p_1\cdots p_k$, for $k\geq 2$. Now, after a calculation we obtain that $t_{n,p_i}=\frac{n}{\gcd(n,p_i)}=\frac{n}{p_i}$ and
$\lambda_{p_i}=(-1)^{k-1}(p_i-1)$,  for $1\leq i\leq k$. Furthermore, for $1\leq i,j\leq
k$ and $i\neq j$, it is clear that $\lambda_{p_i}\neq
\lambda_{p_j}$. If $k\geq 3$,  we have at least three distinct
values among the eigenvalues $\lambda_{p_1}, \lambda_{p_2},\ldots,
\lambda_{p_k}$ and together with the regularity $\lambda_0$ and
$\lambda_1=\mu(n)=(-1)^k$, we conclude that $X_n$ has at least five
eigenvalues, which is a contradiction.

For $k=2$, $n=p_1p_2$ and $\gcd (j,n)=1$, we see that
$t_{n,j}=p_1p_2$ and $\lambda_j=1$. If $\gcd(j,n)=p_1$ then
$t_{n,j}=p_2$ and $\lambda_j=-(p_1-1)$. Similarly, if
$\gcd(j,n)=p_2$ then $\lambda_j=-(p_2-1)$. Thus, we conclude that
together with the regularity $\lambda_0$, $X_n$ has exactly four
distinct eigenvalues.

 \qed
\end{proof}

\begin{theorem}
\label{thm:4 eigenvalues-line of ucg} Let $X_n$ be a unitary Cayley
graph of the order $n$. Then the line graph of $X_n$ has exactly four
eigenvalues if and only if either $n=2p$ or $n=p^{\alpha}$, for some
prime number $p\geq 3$ and $\alpha\geq 2$.
\end{theorem}
\begin{proof}
Suppose that $L(X_n)$ has exactly four distinct eigenvalues. According to
Theorem \ref{lg-regular}, it holds that $X_n$ has either three or
four distinct eigenvalues. Now, suppose that $X_n$ has four distinct
eigenvalues. According to Theorem \ref{thm:ucg-4 eigenvalues}, $n$ is
equal to $p_1p_2$ for some primes $p_1$ and $p_2$. 
Using Theorem \ref{thm:ucg-4 eigenvalues}, we show that the spectrum of $X_n$ is equal to $\{(p_1-1)(p_2-1), -(p_1-1), -(p_2-1),1\}$, for the eigenvalues
$\lambda_i$ of $L(X_n)$, we see $\lambda_i\in\{2(p_1-1)(p_2-1)-2,
(p_1-1)(p_2-2)-2, (p_1-2)(p_2-1)-2, (p_1-1)(p_2-1)-1, -2\}$, according
Theorem \ref{lg-regular}. 
As
$p_1<p_2$, $L(X_n)$ can have four distinct eigenvalues if and only if
$(p_1-2)(p_2-1)-2=-2$ and $p_1=2$.

If $X_n$ is strongly regular, then $n=p^{\alpha}$, for some prime $p$, $\alpha\geq 2$,
the eigenvalues of $X_n$ are $\{p^{\alpha-1}(p-1), 0,
-p^{\alpha-1}\}$ (according to Theorem \ref{thm:sr-line of ucg}).
From Theorem \ref{lg-regular}, we get that any eigenvalue of $L(X_n)$ takes one of the following values
$\{2p^{\alpha-1}(p-1)-2, p^{\alpha-1}(p-1)-2, p^{\alpha-1}(p-2)-2,
-2 \}$. Therefore, $L(X_n)$ has exactly four distinct eigenvalues
only if $p^{\alpha-1}(p-2)-2\neq -2$, i.e. $p\neq 2$.

\qed
\end{proof}

\bigskip

We will show that the newly founded classes of graphs with four eigenvalues (determined in Theorem \ref{thm:4 eigenvalues-line of ucg}) are not circulants. Specifically, we establish that $L(X_{2p})$ and $L(X_{p^{\alpha}})$ are not circulants, except for $L(X_6)$, where $p\geq 3$ and $\alpha\geq 2$.

\medskip

First, we show that $L(X_{2p})$, $p\geq 3$, is not circulant.
According to Theorem \ref{line graphs and circulants} it is
sufficient to prove that $X_{2p}$, $p\geq 3$, is neither the cycle
$C_{2p}$ (with the exception $C_6$), nor complete bipartite graph
$K_{a,b}$ such that $\gcd(a,b)=1$. If $X_{2p}\simeq C_{2p}$ for some
$p\geq 3$, then these graphs have equal number of the edges, so the
equality $\frac{2p\cdot\varphi(2p)}{2}=2p$ holds, which is satisfied only
for $p=3$. In the second case, if $X_{2p}\simeq K_{a,b}$, from the
equality of the number of edges we obtain that $p(p-1)=ab$, whence
without loss of generality we conclude that $p\mid a$. But from the
equality of the orders of the graphs it follows that $a+b=2p$, and
this is true only for $a=b=p$. This is a contradiction since $a$ and
$b$ must be relatively prime numbers.

\medskip

Now, suppose that $L(X_{p^{\alpha}})$, $p\geq 3$ and $\alpha\geq 2$,
is circulant. If $X_{p^{\alpha}}\simeq C_{p^{\alpha}}$ for some $p\geq 3$
and $\alpha\geq 2$, then these graphs have equal number of the
edges, so the equality
$\frac{p^{\alpha}\varphi(p^{\alpha})}{2}=p^{\alpha}$ holds, which is
never satisfied, because the left hand side is always greater than
the right hand side, for $p\geq 3$ and $\alpha\geq 2$. In the second
case, if $X_{p^{\alpha}}\simeq K_{a,b}$, from the equality of the
number of edges we obtain that $\frac{p^{2\alpha-1}(p-1)}{2}=ab$, whence
without loss of generality we conclude that $p^{2\alpha-1}\mid a$,
since $\gcd(a,b)=1$. From the equality of the orders of the graphs
it follows that $a+b=p^{\alpha}$, which is a contradiction since the
left hand side is greater than $p^{2\alpha-1}$.

\bigskip

By the following theorem we find new classes of regular graphs with
four different eigenvalues.

\begin{theorem}
\label{thm:L2-4 eigens} Let $X_n$ be an unitary Cayley graph of the order
$n$. Then $L^2(X_n)$ has exactly four eigenvalues if and only if $n$
is either prime greater than $3$ or a power of two greater than $4$
or equal to $6$.
\end{theorem}
\begin{proof}
Suppose that $L^2(X_n)$ has exactly four distinct eigenvalues. This means
$L(X_n)$ has either three or four distinct eigenvalues.

If $L(X_n)$ is strongly regular, according to the proof of Theorem
\ref{thm:sr-line of ucg} we conclude
that the spectrum of $L(X_n)$ is
$\{2(p-1)-2,p-4,-2\}$, if $n$ is  prime greater than $3$ and
$\{2^{\alpha}-2, 2^{\alpha-1}-2,-2\}$, if $n$
is a power of two greater than $2$. Furthermore, from Theorem \ref{lg-regular} we see that the eigenvalues of $L^2(X_n)$ are
$\{4p-10, 3p-10,2p-8,-2\}$, if $n$ is  prime greater than $3$ and
$\{2^{\alpha+1}-6,3\cdot 2^{\alpha-1}-6, 2^{\alpha}-6,-2\}$, if $n$
is a power of two greater than $2$. Therefore, $L^2(X_n)$ has four
eigenvalues if $n$ is either prime greater than $3$ or a power of
two greater than $4$.

If $L(X_n)$ has four distinct eigenvalues, then we distinguish two
cases depending on the values of $n$. For $n=2p$, where $p\geq 3$ is
prime, the distinct eigenvalues of $L(X_n)$ are \\$\{2p-4, p-4, p-2,
-2\}$, according to Theorem \ref{thm:4 eigenvalues-line of ucg}.
Thus, every eigenvalue $\lambda_i$ of $L^2(X_n)$ satisfies
$\lambda_i\in\{2(2p-4)-2, 3p-10, 3p-8, 2p-8, -2\}$, whence we
conclude that $L^2(X_n)$ has four distinct values only if $p=3$ and,
that is, for $n=6$.
If $n=p^{\alpha}$, for some prime $p\geq 3$ and $\alpha\geq 2$, then
the eigenvalues of $L(X_n)$ are $\{2p^{\alpha-1}(p-1)-2,
p^{\alpha-1}(p-1)-2, p^{\alpha-1}(p-2)-2, -2 \}$, according to Theorem \ref{thm:4 eigenvalues-line of ucg}.
Therefore, every eigenvalue $\lambda_i$ of $L^2(X_n)$ satisfies
$\lambda_i\in\{4p^{\alpha-1}(p-1)-6, \\3p^{\alpha-1}(p-1)-6,
p^{\alpha-1}(3p-4)-6, 2p^{\alpha-1}(p-1)-6, -2\}$. Any two values
from this set are mutually distinct, since $p>3$, whence we conclude
that $L^2(X_n)$ has not four distinct values.

 \qed
\end{proof}

\begin{theorem}
Let $X_n$ be unitary Cayley graph of order $n$. Then $L^3(X_n)$ has
exactly four eigenvalues if and only if $n=6$.
\end{theorem}
\begin{proof}
Suppose that $L^3(X_n)$ has four distinct eigenvalues.
According Theorem \ref{lg-regular}, the graph $L^2(X_n)$ is strongly regular or its spectrum has four distinct eigenvalues.

If
$L^2(X_n)$ is strongly regular, using Theorem \ref{thm:L^2-sr} we have
that $X_n\simeq X_4\simeq C_4$ and $L^k(X_4)\simeq C_4$ (for $k\geq 1$), which
is strongly regular. Using Theorem \ref{thm:L2-4 eigens}, if $L^2(X_n)$ has four different eigenvalues
then either $n$ is prime greater than $3$ or $n$ is a power of $2$
greater than $4$ or $n=6$.

For prime $p>3$, according to the proof of Theorem \ref{thm:L2-4
eigens}, the spectrum of $L^2(X_n)$ is equal to $\{4p-10, 3p-10,2p-8,-2\}$. Moreover, 
every eigenvalue $\lambda_i$ of $L^3(X_n)$ satisfies
$\lambda_i\in\{8p-22, 7p-22, 6p-20, 4p-14, -2\}$. Any two values
from this set are mutually distinct, since $p>3$, whence we conclude
that $L^3(X_p)$ has five distinct eigenvalues.

For $n=2^{\alpha}$ and $\alpha\geq 3$,
according to the proof of
Theorem \ref{thm:L2-4 eigens}, 
the spectrum of $L^2(X_n)$ is equal to $\{2^{\alpha+1}-6, 3\cdot 2^{\alpha-1}-6,2^{\alpha}-6,-2\}$.
furthermore, 
every eigenvalue $\lambda_i$ of
$L^3(X_n)$ satisfies $\lambda_i\in\{2^{\alpha+2}-14,
7\cdot2^{\alpha-1}-14, 3\cdot 2^{\alpha}-14, 2^{\alpha+1}-10, -2\}$.
Any two values from this set are mutually distinct, since
$\alpha\geq 3$, whence we conclude that $L^3(X_n)$ has five distinct eigenvalues.

For $n=6$, according to the proof Theorem  \ref{thm:L2-4 eigens} and Theorem \ref{lg-regular} we see that
the both spectra of the graphs $L^2(X_6)$ and $L^3(X_6)$ are equal to $\{2,-1,1,-2\}$.
 \qed
\end{proof}

It can be easily seen that
the graph $X_6\simeq C_6$ is the only one that can be found at the intersection of the classes mentioned in the Theorems \ref{thm:4 eigenvalues-line of ucg} and \ref{thm:L2-4 eigens}.
Indeed, by
Theorem \ref{thm:4 eigenvalues-line of ucg} we found two classes of graphs that have four distinct values in the spectra:
$L(X_{2p})$, for prime $p\geq 3$, with the set of the eigenvalues
$\{2p-4, p-4, p-2, -2\}$ and $L(X_{p^{\alpha}})$, for prime $p\geq
3$ and $\alpha\geq 2$, with the set of the eigenvalues
$\{2p^{\alpha-1}(p-1)-2, p^{\alpha-1}(p-1)-2, p^{\alpha-1}(p-2)-2,
-2 \}$. On the other hand, by Theorem \ref{thm:L2-4 eigens} we also
found two classes:  $L^2(X_p)$, for prime $p>3$, with the set of the
eigenvalues $\{4p-10, 3p-10,2p-8,-2\}$ and $L^2(X_{2^{\alpha}})$,
for $\alpha\geq 3$, with the set of the eigenvalues
$\{2^{\alpha+1}-6,3\cdot 2^{\alpha-1}-6, 2^{\alpha}-6,-2\}$. If
there are graphs $G_1$ and $G_2$ that belong to the classes obtained by
Theorem \ref{thm:4 eigenvalues-line of ucg} and Theorem
\ref{thm:L2-4 eigens}, respectively, such that $G_1\simeq G_2$, then
they must be cospectral. First, we conclude that $G_1$ has two odd and two even, or three even and one odd
distinct values in its spectrum. On the other hand, 
$G_2$ has four even, or three even and one odd
distinct values in its spectrum.
By comparing the parity of the distinct values in the spectra of
$G_1$ and $G_2$,
we conclude that they have three even and one odd
distinct values in their spectra. This means that
there exit primes
$p\geq 3$, $q>3$ and integer $\alpha\geq 2$ such that $G_1\simeq
L(X_{p^{\alpha}})$ and  $G_2\simeq L^2(X_q)$. In this case, the
regularity and the only odd eigenvalues of the graphs $G_1$ and
$G_2$ must be equal, so the following equalities hold
$2p^{\alpha-1}(p-1)-2=4q-10$ and $ p^{\alpha-1}(p-2)-2=3q-10$. By the substraction of the second
equation from the first, we obtain that $p^{\alpha}=q$, which is
never satisfied, so we obtain a contradiction.

\bigskip

We can show that the classes
of the graphs $L^2(X_p)$ and $L^2(X_{2^{\alpha}})$ (for prime $p>3$
and integer $\alpha\geq 3$), obtained in Theorem
\ref{thm:L2-4 eigens}, are not circulants. According to Theorem
\ref{line graphs and circulants}, we easily see that the classes
$L(X_p)$ and $L(X_{2^{\alpha}})$ neither can be the cycles nor
complete bipartite graphs. Indeed, the  eigenvalues of $L(X_p)$
and $L(X_{2^{\alpha}})$ are $\{2p-4, p-4, -2\}$ and $\{2^{\alpha}-2,
2^{\alpha-1}-2, -2\}$, respectively, and all of them are integers
not equal to $0$. On the other hand, complete bipartite graph of
the order $\frac{p(p-1)}{2}$ has $0$ as an eigenvalue and the cycle
$C_{\frac{p(p-1)}{2}}$ has the irrational eigenvalue $2\cos (\frac{4\pi}{p(p-1)})$,
for $p>3$.

\medskip

\section{Concluding remarks}

In this paper, we establish novel classes of strongly regular graphs and graphs with spectra that comprise four distinct eigenvalues. These classes encompass both circulant and non-circulant connected graphs.
We achieve this by employing specific constructions based on graph operations such as line operations, tensor product, and complement, starting from graphs that possess two or three distinct eigenvalues in their spectra. We further characterize the class of circulant graphs with prime order, including both strongly regular graphs and graphs with spectra exhibiting four distinct eigenvalues.
These findings represent an advancement in the study of characterizing strongly regular circulant graphs, which was initially initiated in \cite{ChRaGu12} and further extended in \cite{Ba22}. It has been proven that circulant graphs with integral spectra that are strongly regular must have composite order. Moreover, it has been noted that the task of classifying the class of integral circulant graphs with four distinct eigenvalues is considerably more challenging compared to characterizing strongly regular integral circulant graphs. This classification likely involves examination of a significantly larger number of cases.
Some proof demonstrations require thorough examination, specifically when dealing with the characterization of circulant graphs possessing three or four eigenvalues. 

However,
the problem of finding a characterization of circulant graphs with four distinct eigenvalues
is still interesting if the discussion is restricted to some special classes of circulant graphs of prescribed order.
Therefore, we believe that it is worthwhile to carry out further investigation in the class of circulant graphs
which are incident graphs of some symmetric block-designs (these graphs have four distinct eigenvalues, as elaborated in \cite{GoRo01}).
Moreover, we think that an effective approach for characterizing such graphs lies in utilizing the statement that a graph is the incident graph of a symmetric block-design if and only if it is a distance-regular graph with a diameter of three \cite{GoRo01}. The classification of all circulant graphs with a prescribed diameter of three represents the most challenging aspect of the proof.  
To illustrate this assertion, we can refer to the proof of Theorem 14 in \cite{Ba15}, which encompasses numerous cases considering very  specific classes of integral circulant graphs.
Moreover, finding graphs with maximal diameter in the class of integral circulant graphs with a prescribed order is a more demanding problem, partially addressed in \cite{Ba23}.  
However, it is possible to characterize certain subclasses of circulant graphs which are incident graphs of symmetric block designs. Specifically, it has been proven that the unitary Cayley graph $X_n$ is the incidence graph of a symmetric block design if and only if $n=2p$, for some $p>2$ (the parameters $(p,p-1,p-2)$ determine the symmetric block design).

 \bigskip {\bf Acknowledgment. }

This research was supported by the research project of the Ministry of Education, Science and Technological Development of the Republic of Serbia (number 451-03-47/2023-01/ 200124).


\begin{thebibliography}{99}




\bibitem{Ba14} 
M. Ba\v si\'c, \textit{Which weighted circulant networks have perfect state transfer?},
Inf. Sciences, Volume 257, (2014), 193--209.

\bibitem{Ba15} 
M. Ba\v si\'c, \textit{Perfect state transfer between non-antipodal vertices in integral circulant graphs},
Ars Comb., 122 (2015), 65--78.


\bibitem{Ba22} 
M. Ba\v si\'c, \textit{Characterization of strongly regular integral circulant graphs by spectral approach},
Appl. Anal. Discrete Math., Volume 16, (2022) 288--306.

\bibitem{Ba23}
M. Ba\v si\'c, A. Ili\'c, A. Stamenkovi\'c, \textit{Maximal diameter of integral circulant graphs}, Inf. Comput,
accepted for publication (2023), https://arxiv.org/abs/2307.09081.




\bibitem{BrHo12}
J. Brown, R. Hoshino, \textit{Line graphs and circulants}, Ars
Combinatoria, 105 (2012), 463--476.

\bibitem{ChRa18}
T. Chelman, S. Raja, \textit{Integral circulant graphs with four distinct eigenvalues},
Discrete Math Algorithms Appl., 10(5) (2018) \#1850062.


\bibitem{ChRaGu12}
T. T. Chelvam, S. Raja, I. Gutman, \textit{Strongly regular integral
circulant graphs and their enegies}, Bull. Inter. Math. Virtual
Inst.   (2012), 9--16.


\bibitem{Cox89}
D. A. Cox, \textit{Primes of the form $x^2 + n y^2$}, New York: Wiley
(1989).

\bibitem{Cv80}
D. Cvetkovi\'c, \textit{A note on construction of graphs by means of
their spectra}, Publications de L'Institut Mathematiquh,  27 (41) (1980),  27--30.



\bibitem{Da95}
E. R. van Dam, \textit{Regular Graphs With Four Eigenvalues}, Linear
Algebra Appl., 226--228 (1995), 139--162.

\bibitem{Da98}
E. R. van Dam, \textit{Nonregular Graphs with Three Eigenvalues},
Journal of Combinatorial Theory, Series B 73(2) (1998), 101--118.

\bibitem{Davis70}
P. J.  Davis, \textit{Circulant Matrices}, Wiley, New York, (1970).


\bibitem{DeDe95}
J. B. Dence, T. P. Dence, \textit{Cubic and quadratic residues modulo a prime}, Mo. J. Math. Sci.,
7(2), (1995).

\bibitem{Ge11}
Y. Ge, B. Greenberg, O. Perez, C. Tamon, \textit{Perfect state transfer, graph products, equitable partitions}, Int J Quantum Inf., 9 (3) (2011), 823--842.


\bibitem{GoRo01}
C. Godsil, G. Royle, \textit{ Algebraic graph theory},
Springer-Verlag New York (2001).

\bibitem{GeSa99}
J.C. George, R.S. Sanders, \textit{When is a Tensor Product of Circulant Graphs
Circulant?}, 	arXiv:math/9907119, (1999).

\bibitem{HardyWright}
G.H. Hardy, E.M. Wright, D.R. Heath-Brown, J.H. Silverman \textit{An introduction to the Theory of
Numbers}, 6th ed, Oxford University Press (2008).



\bibitem{Hasse50}
H. Hasse, \textit{Vorlesungen über Zahlentheorie}, Springer-Verlag Berlin (1950).



\bibitem{Hwang03}
F.K. Hwang, \textit{A survey on multi-loop networks}, Theor Comput
Sci. 299 (2003), 107--121.








\bibitem{KlSa07}
    W. Klotz, T. Sander,
    \textit{Some properties of unitary Cayley graphs},
    Electron. J. Combin. 14 (2007), \#R45.


\bibitem{Sa67}
H. Sachs,  \textit{\"{U}ber Teiler, Faktoren und charakteristische
Polynome von Graphen}, Teil II, Wiss. Z. TH Ilmenau 13 (1967),
405--412.

\bibitem{KlSa10}
W. Klotz, T. Sander, \textit{Integral Cayley graphs over abelian groups}, Electron. J. Comb., 17 (2010)
\#R81.


\bibitem{SaSa15}
J.W. Sander, T. Sander, \textit{On So's conjecture for integral circulant graphs},
Appl. Anal. Discrete Math., Volume 9, (2015) 59--72.

\bibitem{wasin}

W. So, \textit{Integral circulant graphs}, Discrete Math. 306
(2006), 153--158.

\end{thebibliography}
\end{document}